\documentclass[11pt,a4paper]{article}
\usepackage{}
\usepackage{tikz}
\usepackage{amsfonts}
\usepackage{mathrsfs}
\usepackage{multirow}
 \usepackage{epsfig}
 \usepackage{epstopdf}
\usepackage[T1]{fontenc}
\usepackage{geometry}
\usepackage{amsbsy,amsmath,latexsym,amsfonts, epsfig, color, authblk, amssymb, graphics, bm}
\usepackage{epsf,slidesec,epic,eepic}
\usepackage{fancybox}
\usepackage{fancyhdr}
\usepackage{setspace}
\usepackage{nccmath}
\usepackage{cases}
\usepackage{cite}
\usepackage[all]{xy}
\usepackage[colorlinks, citecolor=blue]{hyperref}

\newtheorem{theorem}{Theorem}[section]
\newtheorem{lemma}{Lemma}[section]

\newtheorem{definition}{Definition}[section]
\newtheorem{example}{Example}[section]
\newtheorem{proposition}{Proposition}[section]

\newtheorem{remark}{Remark}[section]
\newtheorem{atheorem}{Theorem}

\newenvironment{proof}{{\noindent \bf Proof:}}{\hfill$\Box$\medskip}

\definecolor{lred}{rgb}{1,0.8,0.8}
\definecolor{lblue}{rgb}{0.8,0.8,1}
\definecolor{dred}{rgb}{0.6,0,0}
\definecolor{dblue}{rgb}{0,0,0.5}
\definecolor{dgreen}{rgb}{0,0.5,0.5}

 \title{Several classes of stationary points for rank regularized minimization problems}
 \author{Yulan Liu\footnote{School of Applied Mathematics, Guangdong University of Technology, Guangzhou.}\ \ {\rm and}\ \
 Shaohua Pan\footnote{Corresponding author\,(shhpan@scut.edu.cn), School of Mathematics, South China University of Technology, Guangzhou.}}

 \date{}

\begin{document}

  \maketitle

 \begin{abstract}
  For the rank regularized minimization problem, we introduce several classes
  of stationary points by the problem itself and its equivalent reformulations
  including the mathematical program with an equilibrium constraint (MPEC),
  the global exact penalty of the MPEC, and the surrogate yielded by eliminating
  the dual part of the exact penalty. A clear relation chart is established
  among these stationary points, which offers a guidance to choose
  an appropriate reformulation for seeking a low-rank solution. As a byproduct,
  for the positive semidefinite (PSD) rank regularized minimization problem,
  we also provide a weaker condition for a local minimizer of its MPEC reformulation
  to be the M-stationary point by characterizing the directional limiting normal cone
  to the graph of the normal cone mapping of the PSD cone.
\end{abstract}



 \section{Introduction}\label{sec1}

 Let $\mathbb{R}^{m\times n}$ be the linear space of all $m\times n\ (m\le n)$ real matrices
 equipped with the trace inner product $\langle \cdot,\cdot\rangle$ and its
 induced norm $\|\cdot\|_F$, i.e., $\langle X,Y\rangle={\rm tr}(X^{\mathbb{T}}Y)$
 for $X,Y\in\mathbb{R}^{m\times n}$. Given a function $f\!:\mathbb{R}^{m\times n}\to\mathbb{R}$,
 we are interested in the rank regularized problem:
 \begin{equation}\label{rank-reg}
  \min_{X\in\mathbb{R}^{m\times n}}F(X):=\nu f(X)+{\rm rank}(X)+\delta_{\Omega}(X)
 \end{equation}
 where $\nu>0$ is the regularization parameter and $\Omega\subseteq\mathbb{R}^{m\times n}$
 is a closed convex set. Unless otherwise stated, we assume that $f$ is locally Lipschitz
 and $\widehat{\partial}f(X)=\partial\!f(X)$ for any $X\in\Omega$, where $\widehat{\partial}f(X)$
 and $\partial\!f(X)$ are the regular and limiting subdifferential of $f$ at $X$, respectively; see Section \ref{sec2.1} for their definitions.
  Such a model is frequently used to seek
 a low-rank matrix under the scenario where a tight estimation is unavailable for the rank of
 the target matrix, and is found to have a host of applications in a variety of fields such as
 statistics \cite{Negahban11}, control and system identification \cite{Fazel03,FPST13},
 signal and image processing \cite{Davenport16}, finance \cite{Pietersz04},
 quantum tomography \cite{Gross11}, and so on.

 \medskip

 Owing to the combinatorial property of the rank function, the problem \eqref{rank-reg}
 is generally NP-hard and it is impossible to achieve a global optimal solution
 by using an algorithm with polynomial-time complexity. So, it is common to obtain
 a desirable local optimal even feasible solution by solving a convex relaxation
 or surrogate problem. Although the nuclear-norm convex relaxation method \cite{Fazel02}
 is very popular, it has a weak ability to promote low-rank solutions and
 even fails to yielding low-rank solutions in some cases \cite{Miao16}.
 After recognizing this deficiency, some researchers pay their attentions to
 the nonconvex surrogates of low-rank optimization problems such as the log-determinant
 surrogate (see \cite{Fazel03,Mohan12}) and the Schatten $p\ (0<p<1)$-norm surrogate \cite{LXW13}.
 As illustrated in \cite{Nie12}, the efficiency of nonconvex surrogates
 depends on its approximation effect.

 \medskip

  Recently, by the variational characterization of the rank function, the authors
  of \cite{BiPan17,LiuBiPan18} reformulated the rank regularized problem \eqref{rank-reg}
  as an equivalent MPEC and derived an equivalent surrogate from its global
  exact penalty. In order to illustrate this, let $\mathscr{L}$ denote the family of
  proper lower semi-continuous (lsc) convex functions
  $\phi\!:\mathbb{R}\to(-\infty,+\infty]$ with
  \begin{equation}\label{phi-assump}
   {\rm int}({\rm dom}\,\phi)\supseteq[0,1],\ \phi(1)=1>t^*:=\mathop{\arg\min}_{0\le t\le 1}\phi(t)
   \ \ {\rm and}\ \ \phi(t^*)=0,
  \end{equation}
  and for each $\phi\in\!\mathscr{L}$ let $\psi\!:\mathbb{R}\to(-\infty,+\infty]$ be
 the associated lsc convex function given by
 \begin{equation}\label{phi-psi}
   \psi(t):=\!\left\{\!
              \begin{array}{cl}
                  \phi(t) &\textrm{if}\ t\in [0,1],\\
                   +\infty & \textrm{otherwise}.
                 \end{array}\right.
  \end{equation}
 With $\phi\in\mathscr{L}$, the rank regularized problem \eqref{rank-reg} can be equivalently reformulated as
 \begin{align}\label{MPEC}
  &\min_{X,W\in\mathbb{R}^{m\times n}}\nu f(X)+{\textstyle\sum_{i=1}^m}\phi(\sigma_i(W))+\delta_{\Omega}(X)\nonumber\\
  &\quad\ \ {\rm s.t.}\ \ \|X\|_*-\langle W,X\rangle=0,\,\|W\|\le 1
 \end{align}
 which is a matrix MPEC since the constraints $\|X\|_*-\langle W,X\rangle=0$
 and $\|W\|\le 1$ are equivalent to $X\in\mathcal{N}_{\mathbb{B}}(W)$ with
 \(
    \mathbb{B}\!:=\{Z\in\mathbb{R}^{m\times n}\,| \, \|Z\|\le 1\},
  \)
  i.e.,
 the optimality condition of $W\in\mathop{\arg\max}_{Z\in \mathbb{B}}\langle X,Z\rangle$.
 Under a mild condition, it was shown in \cite{BiPan17,LiuBiPan18} that
 the following penalized problem
 \begin{align}\label{Epenalty-MPEC}
  &\min_{X,W\in\mathbb{R}^{m\times n}}\nu f(X)+{\textstyle\sum_{i=1}^m}\phi(\sigma_i(W))
          +\rho(\|X\|_*-\langle W,X\rangle)\nonumber\\
  &\quad\ \ {\rm s.t.}\ \ X\in\Omega,\,\|W\|\le 1
 \end{align}
  is a global exact penalty of the MPEC \eqref{MPEC} in the sense that there exists $\overline{\rho}>0$
  such that the problem \eqref{Epenalty-MPEC} associated to each $\rho\ge\overline{\rho}$
  has the same global optimal solution set as \eqref{MPEC} does. With
  the conjugate function $\psi^*(s):=\sup_{t\in\mathbb{R}}\big\{st-\psi(t)\big\}$
  of $\psi$, one may eliminate the dual variable $W$ in \eqref{Epenalty-MPEC}
  and get the following equivalent surrogate of the problem \eqref{rank-reg}
  \begin{equation}\label{surrogate}
  \min_{X\in\Omega}\Big\{\nu f(X)+\rho\|X\|_*-{\textstyle\sum_{i=1}^m}\psi^*(\rho\sigma_i(X))\Big\}.
 \end{equation}

 As well known, when an algorithm is applied to nonconvex and nonsmooth
 optimization problems, one generally expects to achieve a stationary point,
 while the stationary points of equivalent reformulations may have a big difference.
 Thus, it is necessary to clarify the relation among the stationary points
 of \eqref{rank-reg} defined by its equivalent reformulations.
 Moreover, such a clarification is prerequisite to describe the landscape
 of stationary points for the rank regularized problem \eqref{rank-reg}.
 Motivated by this, in Section \ref{sec3} we introduce the R(egular)-stationary
 point, the M-stationary point, the EP-stationary point and the DC-stationary point
 by the problem \eqref{rank-reg} itself and its reformulation \eqref{MPEC}-\eqref{surrogate},
 respectively, and explore the relation among the four classes of stationary points.
 Figure 1 in Section \ref{sec3} shows that the set of M-stationary points
 is almost same as that of R-stationary points, the latter includes that of
 EP-stationary points under a rank condition, and the set of EP-stationary points
 coincides with that of DC-stationary points for some appropriate $\phi$.
 As a byproduct, for the PSD rank regularized minimization problem,
 we also provide a weaker condition than the one in \cite{DingSY14}
 for a local minimizer of its MPEC reformulation to be the M-stationary point,
 by the directional limiting normal cone to the graph of the normal cone mapping
 of the PSD cone $\mathbb{S}_{+}^n$.

 \medskip

 We notice that some active research has been done for the stationary
 points of zero-norm constrained optimization problems (see, e.g.,
 \cite{Burdakov16,PanLN17,Flegel05}); for example,
 Burdakov et al. \cite{Burdakov16} discussed the relation between
 the M-stationary point and the $S$-stationary point of their equivalent
 MPEC reformulation; and Pan et al. \cite{PanLN17} characterizes
 the first-order optimality condition which actually defines a class of
 stationary points by the tangent cone to the zero-norm constrained set.
 To the best of our knowledge, there are few works to study
 the stationary points of rank regularized optimization problems.
 For the special case $\Omega\subseteq\mathbb{S}_{+}^n$, the rank regularized
 problem \eqref{rank-reg} can reduce to a mathematical program with semidefinite
 conic complementarity constraints (MPSCCC) and Ding et al. \cite{DingSY14} have
 established the connection among several class of stationary points for the MPSCCC,
 which are defined by the equivalent reformulations of the complementarity constraints.
 However, this work is concerned with the relation among the stationary points defined by
 different equivalent reformulations of the rank regularized problem \eqref{rank-reg},
 and aims to establish a clear relation chart for these stationary points so that
 the user can be guided to choose an appropriate reformulation to seek a low-rank solution.

 \section{Notation and preliminaries}\label{sec2}

 Throughout this paper, a hollow capital means a finite
 dimensional vector space equipped with the inner product
 $\langle \cdot,\cdot\rangle$ and its induced norm $\|\cdot\|$.
 The notation $\mathbb{S}^n$ denotes the vector space of all $n\times n$
 real symmetric matrices equipped with the Frobenius norm, and $\mathbb{S}_{+}^n$
 means the set of all positive semidefinite matrices in $\mathbb{S}^n$.
 Let $\mathbb{O}^{m\times n}$ be the set of $m\times n$ matrices with orthonormal columns
 and denote $\mathbb{O}^{m\times m}$ by $\mathbb{O}^{m}$. For a given $X\in\mathbb{R}^{m\times n}$,
 we denote by $\|X\|_*$ and $\|X\|$ the nuclear norm and the spectral norm of $X$, respectively,
 and by $\sigma(X)\in\mathbb{R}^m$ the singular value vector arranged in a nonincreasing order;
 and write $\mathbb{O}^{m,n}(X):=\{(U,V)\in\mathbb{O}^{m}\times\mathbb{O}^n\,|\, X=U{\rm Diag}(\sigma(X))V^{\mathbb{T}}\}$.
 For a given $X\in\mathbb{R}^{m\times n}$ and two index sets $\alpha\subseteq\{1,\ldots,m\}$
 and $\beta\subseteq\{1,\ldots,n\}$, $X_{\alpha\beta}$ means the submatrix consists of
 those entries $X_{ij}$ with $i\in\alpha$ and $j\in\beta$. We denote by $E$ and $e$
 the matrix and the vector of all ones respectively whose dimension are known
 from the context, and by $I$ an identity matrix whose dimension is known from the context.
 For a given set $S$, $\delta_{S}$ denotes the indicator function of $S$,
 i.e., $\delta_{S}(x)=0$ if $x\in S$, otherwise $\delta_{S}(x)=+\infty$.
 For a given vector space $\mathbb{Z}$, $\mathbb{B}_{\mathbb{Z}}$ denotes
 the closed unit ball centered at the origin of $\mathbb{Z}$,
 and $\mathbb{B}_{\delta}(z)$ means the closed ball of
 radius $\delta$ centered at $z\in\mathbb{Z}$.
%
%


 \subsection{Normal cones and generalized differentials}\label{sec2.1}

  Let $S\subset\mathbb{Z}$ be a given set. The regular normal cone
  to $S$ at a point $\overline{z}\in S$ is defined by
 \[
  \widehat{\mathcal{N}}_{S}(\overline{z})
   :=\Big\{v\in\mathbb{Z}\ |\ \limsup_{z\xrightarrow[S]{}\overline{z}}
   \frac{\langle v,z-\overline{z}\rangle}{\|z-\overline{z}\|}\le 0\Big\}
 \]
 where the symbol $z\xrightarrow[S]{}\overline{z}$ signifies $z\to\overline{z}$ with $z\in S$,
 while the limiting normal cone to $S$ at $\overline{z}$ is defined as the outer limit
 of $\widehat{\mathcal{N}}_{S}(z)$ as $z\xrightarrow[S]{}\overline{z}$, i.e.,
 \begin{align}\label{NormDef}
   \mathcal{N}_{S}(\overline{z}):=\Big\{v\in \mathbb{Z}\,|\, \exists\, z^k\xrightarrow[S]{}\overline{z},v^k\to v {\ \ \rm with \ \ } v^k\in \widehat{\mathcal{N}}_{S}(z^k)\Big\}.
  \end{align}
  The limiting normal cone $\mathcal{N}_{S}(\overline{z})$ is generally
  not convex, but the regular normal $\widehat{\mathcal{N}}_{S}(\overline{z})$
  is always closed convex which is the negative polar of the contingent cone to
  $S$ at $\overline{x}$:
 \[
   \mathcal{T}_{S}(\overline{z}):=\big\{h\in\mathbb{Z}\ |\ \exists\, t_k\downarrow 0,\,
   h^k\to h\ {\rm with}\ \overline{z}+t_kh^k\in S\big\}.
 \]
 When $S$ is convex, $\mathcal{N}_{S}(\overline{z})$ and $\mathcal{\widehat{N}}_{S}(\overline{z})$
 are the normal cone in the sense of convex analysis \cite{Roc70}.
 The directional limiting normal cone to $S$ at $\overline{z}$ in a direction $u\in\mathbb{X}$
 is defined by
 \begin{equation*}\label{Dir-LNcone}
   \mathcal{N}_{S}(\overline{z};u):=\Big\{z^*\in\mathbb{Z}\ |\ \exists\, t_k\downarrow 0,\,u^k\to u,z^{k*}\to z^*\ {\rm with}\
   z^{k*}\in\widehat{\mathcal{N}}_{S}(\overline{z}\!+\!t_ku^k)\Big\}.
 \end{equation*}
 By comparing with the definition of $\mathcal{N}_{S}(\overline{z})$,
 clearly, $\mathcal{N}_{S}(\overline{z};u)\subseteq \mathcal{N}_{S}(\overline{z})$ for any $u\in\mathbb{X}$.

 \medskip

 Let $g\!: \mathbb{Z}\to[-\infty,+\infty]$ be an extended real-valued lsc function
 with $g(\overline{z})$ finite. The regular subdifferential of $g$ at $\overline{z}$,
 denoted by $\widehat{\partial}g(\overline{z})$, is defined as
  \[
    \widehat{\partial}g(\overline{z}):=\bigg\{z^*\in\mathbb{X}\ \big|\
    \liminf_{z\to \overline{z}\atop z\ne\overline{z}}
    \frac{g(z)-g(\overline{z})-\langle z^*,z-\overline{z}\rangle}{\|z-\overline{z}\|}\ge 0\bigg\};
  \]
  and the (limiting) subdifferential of $g$ at $\overline{z}$, denoted by $\partial g(\overline{z})$, is defined as
  \begin{align}\label{subgradientDef}
   \partial g(\overline{z})=\Big\{ z^*\in \mathbb{X}\,|\, \exists\, z^k\xrightarrow[g]{} z, z^{k,*}\to z^*\ {\rm \, such \, that\, }\ z^{k,*}\in \widehat{\partial}g(z^k)\Big\}.
  \end{align}
  From \cite[Theorem 8.9]{RW98} we know that there is close relation
  between the subdifferentials of $g$ at $\overline{z}$ and the normal cones of its epigraph
  at $(\overline{z},g(\overline{z}))$. Also, from \cite[Exercise 8.14]{RW98},
  \[
    \widehat{\mathcal{N}}_{S}(z)=\widehat{\partial}\delta_{S}(z)
    \ \ {\rm and}\ \
   \mathcal{N}_{S}(z)=\partial\delta_{S}(z)\ \ {\rm for}\ z\in S.
  \]
  In the sequel, we call a point $z$ at which $0\in\partial g(z)$
  (respectively, $0\in\widehat{\partial}g(z)$) is called a limiting
  (respectively, regular) critical point of $g$. By \cite[Theorem 10.1]{RW98},
  a local minimizer of $g$ is necessarily a regular critical point of $g$,
  and then a limiting critical point.
 \subsection{Lipschitz-like properties of multifunctions}\label{sec2.2}

 Let $\mathcal{F}\!:\mathbb{Z}\rightrightarrows\mathbb{W}$ be a given multifunction.
 Consider an arbitrary point $(\overline{z},\overline{w})\in{\rm gph}\mathcal{F}$
 at which $\mathcal{F}$ is locally closed, where ${\rm gph}\mathcal{F}$ denotes
 the graph of $\mathcal{F}$. We recall from \cite{RW98,DR09} the concepts of
 the Aubin property, calmness and metric subregularity of $\mathcal{F}$.
\begin{definition}\label{Aubin-def}
  The multifunction $\mathcal{F}$ is said to have the Aubin property at
  $\overline{z}$ for $\overline{w}$ with modulus $\kappa>0$, if there exist
  $\varepsilon>0$ and $\delta>0$ such that for all $z,z'\in\mathbb{B}_{\varepsilon}(\overline{z})$,
  \[
    \mathcal{F}(z)\cap\mathbb{B}_{\delta}(\overline{w})\subseteq\mathcal{F}(z')
    +\kappa\|z-z'\|\mathbb{B}_{\mathbb{W}}.
  \]
 \end{definition}
 \begin{definition}\label{calm-def}
  The multifunction $\mathcal{F}$ is said to be calm at $\overline{z}$ for $\overline{w}$
  with modulus $\kappa>0$ if there exist $\varepsilon>0$ and $\delta>0$ such that for all
  $z\in\mathbb{B}_{\varepsilon}(\overline{z})$,
  \begin{equation*}\label{calm-inclusion1}
    \mathcal{F}(z)\cap\mathbb{B}_{\delta}(\overline{w})
    \subseteq\mathcal{F}(\overline{z})+\kappa\|z-\overline{z}\|\mathbb{B}_{\mathbb{W}}.
  \end{equation*}
  If in addition $\mathcal{F}(\overline{z})\cap\mathbb{B}_{\delta}(\overline{w})=\{\overline{w}\}$,
  $\mathcal{F}$ is said to be isolated calm at $\overline{z}$ for $\overline{w}$.
 \end{definition}

  By \cite[Exercise 3H.4]{DR09}, the restriction on $z\in\mathbb{B}_{\varepsilon}(\overline{z})$
  in Definition \ref{calm-def} can be removed. It is easily seen that the calmness of
  $\mathcal{F}$ is a ``one-point'' variant of the Aubin property, and the calmness of
  $\mathcal{F}$ at $(\overline{z},\overline{w})\in{\rm gph}\mathcal{F}$
  is implied by its Aubin property or isolated calmness at this point.
  Notice that the calmness of $\mathcal{F}$ at $\overline{z}$ for $\overline{w}\in\mathcal{F}(\overline{z})$
  is equivalent to the metric subregularity of $\mathcal{F}^{-1}$ at $\overline{w}$ for $\overline{z}\in\mathcal{F}^{-1}(\overline{w})$ by \cite[Theorem 3H.3]{DR09}.

  \medskip

  The coderivative and graphical derivative of $\mathcal{F}$ are an convenient
  tool to characterize the Aubin property and the isolated calmness of $\mathcal{F}$,
  respectively. Recall from \cite{RW98} that the coderivative
  of $\mathcal{F}$ at $\overline{z}$ for $\overline{w}$ is the mapping $D^*\mathcal{F}(\overline{z}|\overline{w})\!:\mathbb{W}\rightrightarrows\mathbb{Z}$
  defined by
  \[
    u\in D^*\mathcal{F}(\overline{z}|\overline{w})(v)\Longleftrightarrow
    (u,-v)\in\mathcal{N}_{{\rm gph}\,\mathcal{F}}(\overline{z},\overline{w}),
  \]
  and the graphical derivative of $\mathcal{F}$ at
  $\overline{z}$ for $\overline{w}$ is the mapping
  $D\mathcal{F}(\overline{z}|\overline{w})\!:\mathbb{Z}\rightrightarrows\mathbb{W}$ given by
  \[
    v\in D\mathcal{F}(\overline{z}|\overline{w})(u)\Longleftrightarrow
    (u,v)\in\mathcal{T}_{{\rm gph}\,\mathcal{F}}(\overline{z},\overline{w}).
  \]
  \vspace{-0.8cm}
  \begin{lemma}\label{chara-Aubin}(See \cite[Theorem 5.7]{Mordu93}
  or \cite[Theorem 9.40]{RW98})\ Suppose that $\mathcal{F}$ is locally
  closed at $(\overline{z},\overline{w})$. Then $\mathcal{F}$ has
  the Aubin property at $\overline{z}$ for $\overline{w}$
  iff $D^*\mathcal{F}(\overline{z}|\overline{w})(0)=\{0\}$.
  \end{lemma}
  \begin{lemma}\label{chara-icalm}(See \cite[Proposition 2.1]{KR92}
  or \cite[Proposition 4.1]{Levy96})\ Suppose that $\mathcal{F}$ is locally
  closed at $(\overline{z},\overline{w})$. Then $\mathcal{F}$ is isolated calm
  at $\overline{z}$ for $\overline{w}$ iff $D\mathcal{F}(\overline{z}|\overline{w})(0)=\{0\}$.
  \end{lemma}
 \subsection{Coderivative of the subdifferential mapping $\partial\|\cdot\|_*$}\label{sec2.3}

  For a given $X\in\mathbb{R}^{m\times n}$ with SVD as
  $U[{\rm Diag}(\sigma(X))\ \ 0]V^{\mathbb{T}}$, by \cite[Example 2]{Watson92}
  we have
 \begin{equation}\label{subdiff-nnorm}
 \partial \|X\|_*=\Big\{[U_1\ \ U_2]\left[\begin{matrix}
                                       I &0\\
                                       0 & Z
                                       \end{matrix}\right]
                                       [V_1\ \ V_2]^{\mathbb{T}}\;|\; \|Z\|\leq 1\Big\}
 \end{equation}
 where $U_1$ and $V_1$ are the submatrix consisting of the first $r={\rm rank}(X)$ columns
 of $U$ and $V$, respectively, and $U_2$ and $V_2$ are the submatrix consisting
 of the last $m\!-r$ columns and $n-r$ columns of $U$ and $V$, respectively.
 In this part we recall from \cite{LiuPan19} the coderivative of the subdifferential
 mapping $\partial\|\cdot\|_*$. For this purpose, in the sequel for two positive integers
 $k_1$ and $k_2$ with $k_2\ge k_1$, we denote by $[k_1,k_2]$ the set $\{k_1, k_1\!+\!1,\ldots,k_2\}$.
 For a given $\overline{Z}\in\mathbb{R}^{m\times n}$, define the following index sets associated to
 its singular values:
 \begin{subequations}
 \begin{align}\label{abg-index1}
   \!\alpha\!:=\!\{i\!\in\![1,m]\ |\ \sigma_i(\overline{Z})\!>\!1\},\
   \beta\!:=\!\{i\!\in\![1,m]\ |\ \sigma_i(\overline{Z})\!=\!1\},\ c=[m\!+\!1,n],\qquad\\
   \gamma:=\gamma_1\cup\gamma_0\ {\rm for}\ \gamma_1\!:=\big\{i\in [1,m]\ |\ 0<\sigma_i(\overline{Z})<1\big\},
   \gamma_0\!:=\!\big\{i\in [1,m]\ |\ \sigma_i(\overline{Z})=0\big\},
   \label{abg-index2}
  \end{align}
 \end{subequations}
  and let $\Omega_1,\Omega_2\in\mathbb{S}^{m}$ and $\Omega_3\in\mathbb{R}^{m\times (n-m)}$ be
  the matrices associated to $\sigma(\overline{Z})$ given by
  \begin{subequations}
  \begin{align}\label{Omega1}
   \big(\Omega_1\big)_{ij}
    &:=\left\{\begin{array}{cl}
          \frac{\min(1,\sigma_i(\overline{Z}))-\min(1,\sigma_j(\overline{Z}))}
          {\sigma_i(\overline{Z})-\sigma_j(\overline{Z})}& {\rm if}\ \sigma_i(\overline{Z})\ne\sigma_j(\overline{Z}),\\
           0 & {\rm otherwise}\\
   \end{array}\right.\ \ i,j\in\{1,2,\ldots,m\},\\
   \label{Omega2}
   \big(\Omega_2\big)_{ij}
   &:=\left\{\begin{array}{cl}
          \frac{\min(1,\sigma_i(\overline{Z}))+\min(1,\sigma_j(\overline{Z}))}
          {\sigma_i(\overline{Z})+\sigma_j(\overline{Z})}& {\rm if}\ \sigma_i(\overline{Z})+\!\sigma_j(\overline{Z})\ne 0,\\
           0 & {\rm otherwise}\\
   \end{array}\right.\ \ i,j\in\{1,2,\ldots,m\},\\
   \label{Omega3}
   \big(\Omega_3\big)_{ij}
   &:=\left\{\begin{array}{cl}
          \frac{\min(1,\sigma_i(\overline{Z}))}{\sigma_i(\overline{Z})}
          & {\rm if}\ \sigma_i(\overline{Z})\ne 0,\\
           0 & {\rm otherwise}\\
   \end{array}\right.\ \ i\in\{1,\ldots,m\},j\in\{1,\ldots,n\!-\!m\}.
  \end{align}
  \end{subequations}
  With the matrices $\Omega_1,\Omega_2\in\mathbb{S}^{m}$ and $\Omega_3\in\mathbb{R}^{m\times (n-m)}$,
  we define the following matrices
  \begin{subequations}
  \begin{equation*}\label{Theta}
  \Theta_1:=\left[\begin{matrix}
            0_{\alpha\alpha} & 0_{\alpha \beta} &(\Omega_1)_{\alpha\gamma}\\
            0_{\beta\alpha}& 0_{\beta\beta}&E_{\beta\gamma}\\
            (\Omega_1)_{\gamma\alpha}& E_{\gamma\beta}&E_{\gamma\gamma}\\
           \end{matrix}\right],\
  \Theta_2:=\left[\begin{matrix}
             E_{\alpha\alpha} & E_{\alpha \beta} &E_{\alpha\gamma}\!-\!(\Omega_1)_{\alpha\gamma}\\
             E_{\beta\alpha}& 0_{\beta\beta}&0_{\beta\gamma}\\
             E_{\gamma\alpha}\!-\!(\Omega_1)_{\gamma\alpha}& 0_{\gamma\beta}&0_{\gamma\gamma}\\
            \end{matrix}\right],\\
  \end{equation*}
 \begin{equation*}\label{Sigma}
  \Sigma_1\!:=\!\left[\begin{matrix}
             (\Omega_2)_{\alpha\alpha} & (\Omega_2)_{\alpha\beta} & (\Omega_2)_{\alpha\gamma}\\
             (\Omega_2)_{\beta\alpha}& 0_{\beta\beta}&E_{\beta\gamma}\\
             (\Omega_2)_{\gamma\alpha}& E_{\gamma \beta}&E_{\gamma\gamma}\\
             \end{matrix}\right],
  \Sigma_2\!:=\!\left[\begin{matrix}
              E_{\alpha\alpha}\!-\!(\Omega_2)_{\alpha\alpha} & E_{\alpha \beta}\!-\!(\Omega_2)_{\alpha\beta}
              &E_{\alpha\gamma}\!-\!(\Omega_2)_{\alpha\gamma}\\
              E_{\beta\alpha}\!-\!(\Omega_2)_{\beta\alpha }& 0_{\beta \beta}&0_{\beta\gamma}\\
              E_{\gamma\alpha}\!-\!(\Omega_2)_{\gamma\alpha}& 0_{\gamma\beta}&0_{\gamma\gamma}\\
             \end{matrix}\!\right].
  \end{equation*}
  \end{subequations}
 For the index set $\beta$, we denote the set of all partitions of $\beta$ by
 $\mathscr{P}(\beta)$. Define the set
 \[
   \mathbb{R}_{>}^{|\beta|}:=\big\{z\in\mathbb{R}^{|\beta|}\!:\ z_1\ge\cdots\ge z_{|\beta|}>0\big\}.
 \]
 For any $z\in\mathbb{R}_{>}^{|\beta|}$, let $D(z)\in\mathbb{S}^{|\beta|}$ denote
 the first generalized divided difference matrix of $h(t)=\min(1,t)$ at $z$,
 which is defined as
 \begin{equation}\label{Ddiff-matrix}
   (D(z))_{ij}:=\left\{\begin{array}{cl}
               \!\frac{\min(1,z_i)-\min(1,z_j)}{z_i-z_j}\in[0,1]&{\rm if}\ z_i\ne z_j,\\
                 0 &{\rm if}\ z_i=z_j\ge 1,\\
                1 & {\rm otherwise}.
                \end{array}\right.
 \end{equation}
 Write
 \(
   \mathcal{U}_{|\beta|}
   :=\big\{\overline{\Omega}\in\mathbb{S}^{|\beta|}\!:\ \overline{\Omega}=\lim_{k\to\infty}D(z^k),\,
   z^k\to e_{|\beta|},\, z^k\in\mathbb{R}_{>}^{|\beta|}\big\}.
 \)
 For each $\Xi_1\in\mathcal{U}_{|\beta|}$, by equation \eqref{Ddiff-matrix} there exists
 a partition $(\beta_{+},\beta_{0},\beta_{-})\in\mathscr{P}(\beta)$ such that
  \begin{equation}\label{Xi1-matrix}
   \Xi_1=\left[\begin{matrix}
          0_{\beta_{+}\beta_{+}}& 0_{\beta_{+}\beta_{0}}&(\Xi_1)_{\beta_{+}\beta_{-}}\\
          0_{\beta_{0}\beta_{+}}& 0_{\beta_{0}\beta_{0}}& E_{\beta_{0}\beta_{-}}\\
          (\Xi_1)_{\beta_{+}\beta_{-}}^{\mathbb{T}}&E_{\beta_{-}\beta_{0}}&E_{\beta_{-}\beta_{-}}
         \end{matrix}\right],
 \end{equation}
  where each entry of $(\Xi_1)_{\beta_{+}\beta_{-}}$ belongs to $[0,1]$.
  Let $\Xi_2$ be the matrix associated to $\Xi_1$:
 \begin{equation}\label{Xi2-matrix}
   \Xi_2=\left[\begin{matrix}
          E_{\beta_{+}\beta_{+}}& E_{\beta_{+}\beta_{0}}& E_{\beta_{+}\beta_{-}}\!-\!(\Xi_1)_{\beta_{+}\beta_{-}}\\
          E_{\beta_{0}\beta_{+}}& 0_{\beta_{0}\beta_{0}}& 0_{\beta_{0}\beta_{-}}\\
          E_{\beta_{-}\beta_{+}}\!-\!(\Xi_1)_{\beta_{+}\beta_{-}}^{\mathbb{T}}&0_{\beta_{-}\beta_{0}}&0_{\beta_{-}\beta_{-}}
         \end{matrix}\right].
 \end{equation}
  Now we are in a position to give the coderivative of
 the subdifferential mapping $\partial\|\cdot\|_*$.
 \begin{lemma}(See \cite[Theorem 3.2]{LiuPan19})\label{coderivative}
  Fix an arbitrary $(X,W)\in\!{\rm gph}\,\partial\|\cdot\|_*$ and let
  $\alpha,\beta,\gamma$ and $c$ be defined by \eqref{abg-index1}-\eqref{abg-index2}
  with $\overline{Z}\!=X\!+W$. Let $(\overline{U},\overline{V})\in\mathbb{O}^{m,n}(\overline{Z})$
  with $\overline{V}=[\overline{V}_1\ \ \overline{V}_2]$ where $\overline{V}_1\in\mathbb{O}^{n\times m}$
  and $\overline{V}_2\in\mathbb{O}^{n\times (n-m)}$,
  and for each $H\in\mathbb{R}^{m\times n}$ write $\widetilde{H}=\overline{U}^{\mathbb{T}}\!H\overline{V}$
  and $\widetilde{H}_1=\overline{U}^{\mathbb{T}}\!H\overline{V}_1$.
  Then, $(G,H)\in\mathcal{N}_{{\rm gph}\,\partial \|\cdot\|_*}(X,W)$
  iff the following relations hold
  \begin{subequations}
  \begin{align}\label{equa1-normal}
  \Theta_1\circ\mathcal{S}(\widetilde{H}_1)+\Theta_2\circ \mathcal{S}(\widetilde{G}_1)
    +\Sigma_1\circ \mathcal{X}(\widetilde{H}_1)+\Sigma_2\circ \mathcal{X}(\widetilde{G}_1)=0,\qquad\qquad\\
    \label{equa2-normal}
   \widetilde{G}_{\alpha c}+(\Omega_3)_{\alpha c}\circ
   (\widetilde{H}_{\alpha c}-\widetilde{G}_{\alpha c})=0,\,\widetilde{H}_{\beta c}=0,\,
   \widetilde{H}_{\gamma c}=0,\qquad\qquad\quad\\
   \label{equa3-normal}
   (\widetilde{G}_{\beta\beta},\widetilde{H}_{\beta\beta})\in\!
   \bigcup_{Q\in\mathbb{O}^{|\beta|}\atop \Xi_1\in\mathcal{U}_{|\beta|}}
   \!\left\{(M,N)\ \bigg|\!\left.\begin{array}{ll}
    \Xi_1\circ\widehat{N}+\Xi_2\circ \mathcal{S}(\widehat{M})+\Xi_2\circ \mathcal{X}(\widehat{N})=0\\
    \quad {\rm with}\ \widehat{N}=Q^{\mathbb{T}}NQ,\,\widehat{M}=Q^{\mathbb{T}}MQ,\\
    \quad Q_{\beta_0}^{\mathbb{T}}MQ_{\beta_0}\preceq 0,\
     Q_{\beta_0}^{\mathbb{T}}NQ_{\beta_0}\succeq 0
   \end{array}\right.\!\right\}
 \end{align}
 \end{subequations}
  where $\mathcal{S}\!:\mathbb{R}^{m\times m}\to\mathbb{S}^m$ and
  $\mathcal{X}\!:\mathbb{R}^{m\times m}\to\mathbb{R}^{m\times m}$ are
  linear the mappings defined by
  \begin{equation}\label{ST-oper}
   \mathcal{S}(Y):=(Y\!+\!Y^{\mathbb{T}})/2\ \ {\rm and}\ \
   \mathcal{X}(Y):=(Y\!-\!Y^{\mathbb{T}})/2
   \quad\ \forall\, Y\in\mathbb{R}^{m\times m},
  \end{equation}
  and the notation ``$\circ$'' denotes the Hardmard product operator of two matrices.
  \end{lemma}
 \section{Four classes of stationary points and their relations}\label{sec3}

  To introduce the four classes of stationary points for the problem \eqref{rank-reg},
  with each $\phi\in\!\mathscr{L}$, we write $\widehat{\phi}(t)\!:=\!\phi(|t|)$
  for $t\in\mathbb{R}$ and $\widehat{\Phi}(x)\!:={\textstyle\sum_{i=1}^m}\widehat{\phi}(x_i)$
  for $x\in\mathbb{R}^m$; and with the associated $\psi$, write $\widehat{\psi}(t)\!:=\!\psi(|t|)$
  for $t\in\mathbb{R}$ and $\widehat{\Psi}(x)\!:={\textstyle\sum_{i=1}^m}\widehat{\psi}(x_i)$
  for $x\in\mathbb{R}^m$. Clearly, $\widehat{\Phi}$ and $\widehat{\Psi}$ are
  absolutely symmetric, i.e., $\widehat{\Phi}(Px)=\widehat{\Phi}(x)$
  and $\widehat{\Psi}(Px)=\widehat{\Psi}(x)$ for any $m\times m$
  signed permutation matrix $P$. Also, $\widehat{\Phi}\circ\sigma$
  is globally Lipschitz continuous over the ball $\mathbb{B}$.
  The following equivalent relations are often used in the subsequent analysis
  \begin{subequations}
  \begin{align}\label{XW-equa1}
   \|X\|_*-\langle X,W\rangle=0,\|W\|\le 1
   \Longleftrightarrow W\in\mathop{\arg\max}_{Z\in\mathbb{B}}\,\langle Z,X\rangle\Longleftrightarrow X\in\mathcal{N}_{\mathbb{B}}(W)\quad\ \\
   \label{XW-equa2}
   \qquad\qquad\qquad\qquad\qquad\qquad\quad\Longleftrightarrow W\in\partial\|\cdot\|_*(X)
   \Longleftrightarrow (X,W)\in{\rm gph}\,\partial\|\cdot\|_*
  \end{align}
 \end{subequations}

 \subsection{R-stationary point}\label{sec3.1}

  Recall that $\overline{X}\in\mathbb{R}^{m\times n}$ is a regular critical point of $F$
  if $0\in\widehat{\partial}F(\overline{X})$. Since the rank function is regular
  by \cite[Lemma 2.1]{WuPanBi18} and \cite[Corollary 7.5]{Lewis05}, by combining
  with the assumption on $f$, from \cite[Corollary 10.9]{RW98} we have
  $\widehat{\partial}F(\overline{X})\supseteq\partial\!f(\overline{X})
   +\partial{\rm rank}(\overline{X})+\mathcal{N}_{\Omega}(\overline{X})$.
  In view of this, we introduce the following R-stationary point of the problem \eqref{rank-reg}.
 \begin{definition}\label{Def-Rspc}
  A matrix $\overline{X}\in\mathbb{R}^{m\times n}$ is called a R-stationary
  point of the problem \eqref{rank-reg} if
  \[
     0\in \nu\partial\!f(\overline{X})+\partial{\rm rank}(\overline{X})
     +\mathcal{N}_{\Omega}(\overline{X}).
  \]
 \end{definition}
 \begin{remark}
  Clearly, every R-stationary point of \eqref{rank-reg} is
  a regular critical point of $F$. By the given assumption on $f$
  and \cite[Exercise 10.10]{RW98}, for any $X\in\Omega$ it holds that
  \[
    \partial\!F(X)\subset \nu\partial\!f(X)+\partial({\rm rank}+\delta_{\Omega})(X).
  \]
  Thus, when $\partial({\rm rank}+\delta_{\Omega})(\overline{X})
  \subset \partial{\rm rank}(\overline{X})+\mathcal{N}_{\Omega}(\overline{X})$,
  the limiting critical point of $F$ is same as its regular critical point,
  and coincides with the R-stationary point of \eqref{rank-reg}.
 \end{remark}
 \subsection{M-stationary point}\label{sec3.1}

  By invoking the relation \eqref{XW-equa2}, clearly, the MPEC \eqref{MPEC}
  can be compactly written as
  \begin{equation}\label{EMPEC1}
    \min_{X,W\in\mathbb{R}^{m\times n}}\Big\{\widetilde{F}(X,W):=\nu f(X)+\widehat{\Phi}(\sigma(W))
    +\delta_{\Omega}(X)+\delta_{{\rm gph}\partial\|\cdot\|_*}(X,W)\Big\}.
  \end{equation}
  Moreover, under a suitable constraint qualification (CQ), the following inclusion holds:
  \[
    \partial\widetilde{F}(X,W)\subseteq \big[\nu \partial\!f(X)+\mathcal{N}_{\Omega}(X)\big]\times
    \partial(\widehat{\Phi}\circ\sigma)(W)+\mathcal{N}_{{\rm gph}\partial\|\cdot\|_*}(X,W).
  \]
  Motivated by this, we introduce the M-stationary point of the problem \eqref{rank-reg} as follows.
 \begin{definition}\label{Def-Mspc}
  A matrix $\overline{X}\in\mathbb{R}^{m\times n}$ is called an M-stationary point of
  the problem \eqref{rank-reg} associated to $\phi\in\mathscr{L}$ if there exist
  $\overline{W}\in\partial\|\cdot\|_*(\overline{X})$
  and $\Delta W\in\partial(\widehat{\Phi}\circ\sigma)(\overline{W})$ such that
  \begin{equation}\label{M-equa}
    0\in \nu \partial\!f(\overline{X})+\mathcal{N}_{\Omega}(\overline{X})
    +D^*\partial\|\cdot\|_*(\overline{X}|\overline{W})(\Delta W).
  \end{equation}
 \end{definition}
 \begin{remark}\label{M-Spoint}
  When $\Omega\subseteq\mathbb{S}_{+}^n$, the rank regularized problem \eqref{rank-reg}
  can be reformulated as
  \begin{align}\label{SDPMPEC}
  &\min_{X,W\in\mathbb{S}^{n}}\nu f(X)+{\textstyle\sum_{i=1}^m}\phi(\sigma_i(W))+\delta_{\Omega}(X)\nonumber\\
  &\quad {\rm s.t.}\ \ \langle I-W,X\rangle=0,\,W\in\mathbb{S}_{+}^n,\, I-W\in\mathbb{S}_{+}^n.
 \end{align}
  Notice that $\langle I-W,X\rangle=0,X\in\mathbb{S}_{+}^n,W\in\mathbb{S}_{+}^n$
  and $I-W\in\mathbb{S}_{+}^n$ iff $(X,W\!-I)\in{\rm gph}\mathcal{N}_{\mathbb{S}^n_+}$
  and $W\in\mathbb{S}_{+}^n$. So, for this case, $\overline{X}\in\Omega$
  is an M-stationary point if and only if there exist $(\overline{Y},\Delta Y)\in\mathbb{S}_{+}^n\times\mathbb{S}^n$
  with $\overline{Y}\!-I\in\mathcal{N}_{\mathbb{S}_{+}^n}(\overline{X})$ and
  $\Delta Y\in\partial(\widehat{\Phi}\circ\sigma)(\overline{Y})+\mathcal{N}_{\mathbb{S}_{+}^n}(\overline{Y})$ such that
  \[
    0\in \nu \partial\!f(\overline{X})+\mathcal{N}_{\Omega}(\overline{X})
    +D^*\mathcal{N}_{\mathbb{S}^n_+}(\overline{X}|\overline{Y}\!-I)(-\Delta Y),
  \]
  or equivalently, there exist $\overline{Y}\!\in\mathbb{S}_{+}^n$
  with $\overline{Y}\!-I\in\mathcal{N}_{\mathbb{S}_{+}^n}(\overline{X})$
  and $(\overline{\Gamma}_1,\overline{\Gamma}_2)\in
  \mathbb{S}^n\times\mathbb{S}^{n}$ such that
 \begin{subnumcases}{}\label{MpointR-equa1}
  0\in \nu\partial\!f(\overline{X})+\overline{\Gamma}_1+\mathcal{N}_{\Omega}(\overline{X}), \\
  \label{MpointR-equa2}
  0\in \partial (\widehat{\Phi}\circ\sigma)(\overline{Y})-\overline{\Gamma}_2+\mathcal{N}_{\mathbb{S}^n_+}(\overline{Y}),\\
  \label{MpointR-equa3}
   \overline{\Gamma}_1\in D^*\mathcal{N}_{\mathbb{S}^n_+}(\overline{X}|\overline{Y}\!-I)(-\overline{\Gamma}_2).
  \end{subnumcases}
 \end{remark}

 For this class of stationary points, we have the following proposition
 that is the key to achieve the relation between the M-stationary point
 and the R-stationary point.
 \begin{proposition}\label{prop-Mspc}
  Let $\mathscr{L}_1$ denote the family of those $\phi\in\!\mathscr{L}$
  that is differentiable on $(0,1]$. If $\overline{X}$ is an M-stationary point
  of the problem \eqref{rank-reg} associated to $\phi\in\!\mathscr{L}_1$,
  then there exist $\overline{W}\in\partial\|\cdot\|_*(\overline{X})$ and $\Delta\Gamma\in\nu\partial\!f(\overline{X})+\mathcal{N}_{\Omega}(\overline{X})$
  such that for the index sets $\alpha,\beta,c,\gamma,\gamma_1$ and $\gamma_0$
  defined as in \eqref{abg-index1}-\eqref{abg-index2} with
  $\overline{Z}=\overline{X}+\overline{W}$ and
  $(\overline{U},\overline{V})\in\mathbb{O}^{m,n}(\overline{Z})$,
  \begin{equation}\label{DetaGamma}
    \Delta \Gamma\!=\overline{U}\!\left[\begin{matrix}
    0_{\alpha\alpha}&0_{\alpha\beta}&0_{\alpha\gamma}&0_{\alpha c}\\
    0_{\beta\alpha} &(\Delta\widetilde{\Gamma})_{\beta\beta}
    &(\Delta\widetilde{\Gamma})_{\beta\gamma}&(\Delta\widetilde{\Gamma})_{\beta c}\\
       0_{\gamma\alpha}&(\Delta\widetilde{\Gamma})_{\gamma\beta}
       &(\Delta\widetilde{\Gamma})_{\gamma\gamma}&(\Delta\widetilde{\Gamma})_{\gamma c}
     \end{matrix}\right]\!\overline{V}^{\mathbb{T}}\ \ {\rm for}\
    \Delta\widetilde{\Gamma}\!=\!\overline{U}^{\mathbb{T}}\!\Delta\Gamma\overline{V}, \mathcal{S}[(\Delta\widetilde{\Gamma})_{\beta\beta}]\!=0.
 \end{equation}
 In particular, if $\phi'(t)\ne 0$ for all $t\in(0,1)$, then $\gamma_1=\emptyset$; and
 if $0\notin\partial\widehat{\phi}(0)$, then $\gamma_0=\emptyset$.
 \end{proposition}
 \begin{proof}
  Let $\overline{X}$ be an M-stationary point of the problem \eqref{rank-reg}
  associated to $\phi\in\!\mathscr{L}_1$. By Definition \ref{Def-Mspc},
  there exist $\overline{W}\in\partial\|\cdot\|_*(\overline{X})$
  and $\Delta W\in\partial(\widehat{\Phi}\circ\sigma)(\overline{W})$ such that
  \eqref{M-equa} holds. So, there exists $\Delta\Gamma\in\nu \partial f(\overline{X})+\mathcal{N}_{\Omega}(\overline{X})$
  such that $-\Delta\Gamma\in D^*\partial\|\cdot\|_*(\overline{X}|\overline{W})(\Delta W)$.
  We argue that $\Delta\Gamma$ has the form of \eqref{DetaGamma}.
  Since $(\overline{U},\overline{V})\in\mathbb{O}^{m,n}(\overline{Z})$,
  from $\overline{W}\in\partial\|\cdot\|_*(\overline{X})$,
  \begin{subequations}
  \begin{align}\label{XbarSVD}
  \overline{X}&=\overline{U}\left[\begin{matrix}
                   {\rm Diag}(\sigma_{\alpha}(\overline{Z})\!-\!e_{\alpha})& 0_{\alpha\beta}&0_{\alpha\gamma}&0_{\alpha c}\\
                   0_{\beta\alpha} &  0_{\beta\beta} &  0_{\beta\gamma}&0_{\beta c}\\
                   0_{\gamma\alpha}&0_{\gamma\beta}  & 0_{\gamma\gamma}&0_{\gamma c}
                   \end{matrix}\right]\overline{V}^{\mathbb{T}},\\
  \label{YbarSVD}
  \overline{W}&=\overline{U}\left[\begin{matrix}
                   I_{\alpha\alpha}& 0_{\alpha\beta}&0_{\alpha\gamma}&0_{\alpha c}\\
                   0_{\beta\alpha} & I_{\beta\beta} &  0_{\beta\gamma}&0_{\beta c}\\
                   0_{\gamma\alpha}&0_{\gamma\beta}  & {\rm Diag}(\sigma_{\gamma}(\overline{Z}))&0_{\gamma c}
                   \end{matrix}\right]\overline{V}^{\mathbb{T}}.
  \end{align}
 \end{subequations}
 Since $\widehat{\Phi}$ is absolutely symmetric and $\Delta W\!\in\partial(\widehat{\Phi}\circ\sigma)(\overline{W})$,
 by \cite[Corollary 2.5]{Lewis95} and equation \eqref{YbarSVD} there exist
 $(\widehat{U},\widehat{V})\in\mathbb{O}^{m,n}(\overline{W})$ and
 $\overline{w}\in\partial\widehat{\Phi}(\sigma(\overline{W}))$ such that
 \begin{equation}\label{DetaW}
   \Delta W=\widehat{U}[{\rm Diag}(\overline{w})\ \ 0]\widehat{V}^{\mathbb{T}}.
 \end{equation}
 From $\overline{w}\in\partial\widehat{\Phi}(\sigma(\overline{W}))$,
 the definition of $\widehat{\Phi}$ and equation \eqref{YbarSVD}, it follows that
 \begin{equation}\label{equa-barw}
   \overline{w}_i=\phi'(1)\ {\rm for}\ i\in\alpha\cup\beta;\ \
   \overline{w}_i=\phi'(\sigma_{i}(\overline{Z}))\ {\rm for}\ i\in\gamma_1;\ \
   \overline{w}_i\in\partial\widehat{\phi}(0)\ {\rm for}\ i\in\gamma_0.
 \end{equation}
 Without loss of generality, we assume that the matrix $\overline{Z}$ has $r$ distinct
 singular values belonging to $(0,1)$. Let $\overline{\mu}_1>\overline{\mu}_2>\cdots>\overline{\mu}_{r}$
 be the $r$ distinct singular values and write
 \[
   a_k\!:=\big\{i\in\gamma_1\ |\ \sigma_i(\overline{Z})\!=\overline{\mu}_k\big\}\ \ {\rm for}\ \
   k=1,2,\ldots,r.
 \]
 Since $(\widehat{U},\widehat{V})\in\mathbb{O}^{m,n}(\overline{W})$,
 from equation \eqref{YbarSVD} and \cite[Proposition 5]{DingST14},
 there exist a block diagonal matrix $\widehat{Q}={\rm Diag}(Q_0,Q_1,\ldots,Q_r)$
 with $Q_0\in\mathbb{O}^{|\alpha|+|\beta|}$ and $Q_k\in \mathbb{O}^{|a_k|}$
 for $k=1,2,\ldots,r$, and orthogonal matrices $Q'\in \mathbb{O}^{|\gamma_0|}$ and
 $Q''\in \mathbb{O}^{|\gamma_0\cup c|}$ such that
 \[
   \widehat{U}=\overline{U}\left[\begin{matrix}
              \widehat{Q} &0\\
              0 & Q'
              \end{matrix}\right]
   \ \ {\rm and}\ \
  \widehat{V}=\overline{V}\left[\begin{matrix}
   \widehat{Q}&0\\
    0 & Q''
    \end{matrix}\right].
 \]
 Together with equations \eqref{DetaW} and \eqref{equa-barw}, it is not difficult to obtain that
 \begin{equation*}
   \Delta W=\overline{U}
   \left[\begin{matrix}
      {\rm Diag}(\overline{w}_{\alpha\cup\beta\cup\gamma_1}) & 0\\
      0 &Q'{\rm Diag}(\overline{w}_{\gamma_0})(Q_{\gamma_0}'')^{\mathbb{T}}
      \end{matrix}\right]\overline{V}^{\mathbb{T}},
 \end{equation*}
 and consequently
 \begin{equation}\label{detaW-equa1}
  \Delta\widetilde{W}\!:=\overline{U}^{\mathbb{T}}\Delta W\overline{V}
   =\left[\begin{matrix}
      {\rm Diag}(\overline{w}_{\alpha\cup\beta\cup\gamma_1}) & 0\\
      0 &Q'{\rm Diag}(\overline{w}_{\gamma_0})(Q_{\gamma_0}'')^{\mathbb{T}}
      \end{matrix}\right].
 \end{equation}
 Since $(-\Delta\Gamma,-\Delta W)\in\!\mathcal{N}_{{\rm gph}\,\partial \|\cdot\|_*}(\overline{X},\overline{W})$,
 by equation \eqref{equa1-normal}-\eqref{equa2-normal} of Lemma \ref{coderivative}, we get
 \begin{subequations}
  \begin{align}\label{normal1-equa1}
  \Theta_1\circ\mathcal{S}(\Delta\widetilde{W}_1)+\Theta_2\circ \mathcal{S}(\Delta\widetilde{\Gamma}_1)
    +\Sigma_1\circ \mathcal{X}(\Delta\widetilde{W}_1)+\Sigma_2\circ \mathcal{X}(\Delta\widetilde{\Gamma}_1)=0,\\
    \label{normal1-equa2}
   (\Delta\widetilde{\Gamma})_{\alpha c}+(\Omega_3)_{\alpha c}\circ
   [(\Delta\widetilde{W})_{\alpha c}-(\Delta\widetilde{\Gamma})_{\alpha c}]=0,\ (\Delta\widetilde{W})_{\beta c}=0,\
   \ (\Delta\widetilde{W})_{\gamma c}=0
 \end{align}
 \end{subequations}
 where $\Delta\widetilde{\Gamma}_1\!:=\overline{U}^{\mathbb{T}}\Delta\Gamma\overline{V}_{\!\alpha\cup\beta\cup\gamma},
 \Delta\widetilde{W}_1\!:=\overline{U}^{\mathbb{T}}\Delta W\overline{V}_{\!\alpha\cup\beta\cup\gamma}$,
 and the matrices $\Theta_1,\Theta_2,\Sigma_1$ and $\Sigma_2$ are defined as in Section \ref{sec2.3}.
 Notice that $[\Delta\widetilde{W}_1]_{\alpha\cup\beta\cup\gamma_1,\alpha\cup\beta\cup\gamma_1}$
 is a diagonal matrix by equation \eqref{detaW-equa1}. Together with \eqref{normal1-equa1}-\eqref{normal1-equa2}
 and \eqref{Omega1}-\eqref{Omega3}, it follows that
  \begin{subequations}
  \begin{align}\label{normal2-equa1}
  (\Delta \widetilde{W})_{\alpha c}=0,\,(\Delta \widetilde{\Gamma})_{\alpha c}=0,\,
  (\Delta \widetilde{W})_{\gamma\gamma}=0,\qquad\\
  \label{normal2-equa11}
  [\mathcal{S}(\Delta\widetilde{\Gamma}_1)]_{\alpha\alpha}
  +(\Sigma_2)_{\alpha\alpha}\circ[\mathcal{X}(\Delta\widetilde{\Gamma}_1)]_{\alpha\alpha}=0,\qquad\\
  \label{normal2-equa2}
  (\Theta_2)_{\alpha\beta}\circ [\mathcal{S}(\Delta\widetilde{\Gamma}_1)]_{\alpha\beta}
  +(\Sigma_2)_{\alpha\beta}\circ [\mathcal{X}(\Delta\widetilde{\Gamma}_1)]_{\alpha\beta}=0,
  \\\label{normal2-equa3}
  (\Theta_2)_{\beta\alpha}\circ [\mathcal{S}(\Delta\widetilde{\Gamma}_1)]_{\beta\alpha}
  +(\Sigma_2)_{\beta\alpha}\circ[\mathcal{X}(\Delta\widetilde{\Gamma}_1)]_{\beta\alpha}=0,\\
  \label{normal2-equa4}
  (\Theta_2)_{\alpha\gamma}\circ[\mathcal{S}(\Delta\widetilde{\Gamma}_1)]_{\alpha\gamma}
  +(\Sigma_2)_{\alpha\gamma}\circ[\mathcal{X}(\Delta\widetilde{\Gamma}_1)]_{\alpha\gamma}=0,\\
  \label{normal2-equa5}
  (\Theta_2)_{\gamma\alpha}\circ[\mathcal{S}(\Delta\widetilde{\Gamma}_1)]_{\gamma\alpha}
  +(\Sigma_2)_{\gamma\alpha}\circ[\mathcal{X}(\Delta\widetilde{\Gamma}_1)]_{\gamma\alpha}=0.
  \end{align}
 \end{subequations}
 Notice that \eqref{normal2-equa11} is equivalent to
 \(
   (E+\Sigma_2)_{\alpha\alpha}(\Delta\widetilde{\Gamma}_1)_{\alpha\alpha}
   +(E-\Sigma_2)_{\alpha\alpha}(\Delta\widetilde{\Gamma}_1^{\mathbb{T}})_{\alpha\alpha}=0
 \)
 which, by the fact that the entries of $\Sigma_{2}$ belongs to $(0,1)$,
 implies that $(\Delta\widetilde{\Gamma}_1)_{\alpha\alpha}=0$.
 Notice that equations (\ref{normal2-equa2}) and (\ref{normal2-equa3}) can be equivalently written as
  \begin{subequations}
  \begin{align}\label{normal3-equa1}
  (\Theta_2+\Sigma_2)_{\alpha\beta}\circ (\Delta\widetilde{\Gamma}_1)_{\alpha\beta}=(\Sigma_2-\Theta_2)_{\alpha\beta}\circ (\Delta\widetilde{\Gamma}_1^{\mathbb{T}})_{\alpha\beta},\\
  \label{normal3-equa2}
  (\Theta_2+\Sigma_2)_{\beta\alpha}\circ (\Delta\widetilde{\Gamma}_1)_{\beta\alpha}=(\Sigma_2-\Theta_2)_{\beta\alpha}\circ (\Delta\widetilde{\Gamma}_1^{\mathbb{T}})_{\beta\alpha}.
  \end{align}
 \end{subequations}
  Since $[(\Delta\widetilde{\Gamma}_1)_{\beta\alpha}]^{\mathbb{T}}=(\Delta\widetilde{\Gamma}_1^{\mathbb{T}})_{\alpha\beta}$
  and $[(\Delta\widetilde{\Gamma}_1^{\mathbb{T}})_{\beta\alpha}]^{\mathbb{T}}=(\Delta\widetilde{\Gamma}_1)_{\alpha\beta}$,
  by imposing the transpose to the both sides of equality (\ref{normal3-equa2}) we immediately obtain that
  \[
  (\Delta\widetilde{\Gamma}_1^{\mathbb{T}})_{\alpha\beta}
  =\big[(\Sigma_2-\Theta_2)_{\alpha\beta}\oslash(\Theta_2+\Sigma_2)_{\alpha\beta}\big]
    \circ (\Delta\widetilde{\Gamma}_1)_{\alpha\beta}
  \]
  where ``$\oslash$'' denotes the entries division operator of two matrices.
  Substituting this equality into \eqref{normal3-equa1} yields that
  $(\Delta\widetilde{\Gamma}_1)_{\alpha\beta}=0$,
  and then $(\Delta\widetilde{\Gamma}_1)_{\beta\alpha}=0$. Similarly, from \eqref{normal2-equa4}
  and \eqref{normal2-equa5}, we can obtain $(\Delta\widetilde{\Gamma}_1)_{\alpha\gamma}=0$ and
  $(\Delta\widetilde{\Gamma}_1)_{\beta\gamma}=0$. Thus,
  \[
    \overline{U}^{\mathbb{T}}\Delta\Gamma \overline{V}=\Delta\widetilde{\Gamma}
    =\left[\begin{matrix}
       0_{\alpha\alpha}&0_{\alpha\beta}&0_{\alpha\gamma}&0_{\alpha c}\\
       0_{\beta\alpha} &(\Delta\widetilde{\Gamma})_{\beta\beta}
       &(\Delta\widetilde{\Gamma})_{\beta\gamma}&(\Delta\widetilde{\Gamma})_{\beta c}\\
       0_{\gamma\alpha}&(\Delta\widetilde{\Gamma})_{\gamma\beta}
       &(\Delta\widetilde{\Gamma})_{\gamma\gamma}&(\Delta\widetilde{\Gamma})_{\gamma c}
     \end{matrix}\right].
  \]
  Thus, to complete the proof of the first part, we only need to argue that
  $\mathcal{S}[(\Delta\widetilde{\Gamma})_{\beta\beta}]=0$.
  Since $(-\Delta\Gamma,-\Delta W)\in\!\mathcal{N}_{{\rm gph}\,\partial \|\cdot\|_*}(\overline{X},\overline{W})$,
  by \eqref{equa3-normal} there exist $Q\in\mathbb{O}^{|\beta|}$ and $\Xi_1\in\mathcal{U}_{|\beta|}$
  having the form \eqref{Xi1-matrix} for some partition $(\beta_{+},\beta_{0},\beta_{-})$
  of $\beta$ such that
  \begin{align}\label{normal4-equa1}
    \Xi_1\circ Q^{\mathbb{T}}(\Delta\widetilde{W})_{\beta\beta}Q
    +\Xi_2\circ \mathcal{S}\big[Q^{\mathbb{T}}(\Delta\widetilde{\Gamma})_{\beta\beta}Q\big]
    +\Xi_2\circ \mathcal{X}\big[Q^{\mathbb{T}}(\Delta\widetilde{W})_{\beta\beta}Q\big]=0,\\
    Q_{\beta_0}^{\mathbb{T}}(\Delta\widetilde{\Gamma})_{\beta\beta}Q_{\beta_0}\succeq 0,\
    Q_{\beta_0}^{\mathbb{T}}(\Delta\widetilde{W})_{\beta\beta}Q_{\beta_0}\preceq 0
    \qquad\qquad\qquad
  \label{normal4-equa2}
  \end{align}
  where the matrix $\Xi_2$ associated with $\Xi_1$ has the form of \eqref{Xi2-matrix}.
  From \eqref{detaW-equa1} and the first equality in \eqref{equa-barw},
  $(\Delta\widetilde{W})_{\beta\beta}={\rm Diag}(\overline{w}_{\beta})=\phi'(1)I$.
  Notice that $\phi'(1)>0$ by \eqref{phi-assump}. We deduce
  $\beta_0=\emptyset$ from the second inequality of \eqref{normal4-equa2}.
  Since $\mathcal{X}[Q^{\mathbb{T}}(\Delta\widetilde{W})_{\beta\beta}Q]=0$,
  \eqref{normal4-equa1} reduces to
  \[
    \Xi_1\circ(Q^{\mathbb{T}}{\rm Diag}(\overline{w}_{\beta})Q)
    +\Xi_2\circ \mathcal{S}\big[Q^{\mathbb{T}}(\Delta\widetilde{\Gamma})_{\beta\beta}Q\big]=0.
  \]
  Since $Q^{\mathbb{T}}{\rm Diag}(\overline{w}_{\beta})Q\succ 0$, by using the expressions of
  $\Xi_1$ and $\Xi_2$ we have $\beta_{-}=\emptyset$, and then the last equality
  reduces to $0=\mathcal{S}\big[Q^{\mathbb{T}}(\Delta\widetilde{\Gamma})_{\beta\beta}Q\big]
  =\mathcal{S}[(\Delta\widetilde{\Gamma})_{\beta\beta}]$. Thus,
  we complete the proof of the first part. By combining $(\Delta \widetilde{W})_{\gamma\gamma}=0$
  with \eqref{detaW-equa1} and \eqref{equa-barw}, it is easy to see that
  if $\phi'(t)\ne 0$ for any $t\in(0,1)$, then $\gamma_1=\emptyset$;
  and if $0\notin\partial\widehat{\phi}(0)$, then $\gamma_0=\emptyset$.
 \end{proof}

 Now we state the relation between the M-stationary point and
 the R-stationary point.
 \vspace{-0.5cm}
 \begin{theorem}\label{prop-MWrelation}
  If $\overline{X}$ is an M-stationary point of the problem \eqref{rank-reg}
  associated to $\phi\in\!\mathscr{L}_1$, then it is also a R-stationary point.
  Conversely, if $\overline{X}$ is a R-stationary point of \eqref{rank-reg},
  then it is an M-stationary point associated to those $\phi\in\!\mathscr{L}$
  with $0\in\partial\widehat{\phi}(0)$.
 \end{theorem}
 \begin{proof}
  Let $\overline{X}$ be an M-stationary point of \eqref{rank-reg}
  associated to $\phi\in\!\mathscr{L}_1$. By Proposition \ref{prop-Mspc},
  there exist $\overline{W}\in\partial\|\cdot\|_*(\overline{X})$ and
  $\Delta\Gamma\in\nu\partial\!f(\overline{X})+\mathcal{N}_{\Omega}(\overline{X})$
  such that for the index sets $\alpha,\beta,c,\gamma,
  \gamma_1,\gamma_0$ defined as in \eqref{abg-index1}-\eqref{abg-index2}
  with $\overline{Z}=\overline{X}+\overline{W}$ and
  $(\overline{U},\overline{V})\in\mathbb{O}^{m,n}(\overline{Z})$,
  the matrix $\Delta\Gamma$ takes the form of \eqref{DetaGamma}. Let
  \(
  \Delta\widetilde{Z}=
  \left[\begin{matrix}
   (\Delta\widetilde{\Gamma})_{\beta\beta} &(\Delta\widetilde{\Gamma})_{\beta\gamma}
   &(\Delta\widetilde{\Gamma})_{\beta c}\\
  (\Delta\widetilde{\Gamma})_{\gamma\beta}&(\Delta\widetilde{\Gamma})_{\gamma\gamma}
  &(\Delta\widetilde{\Gamma})_{\gamma c}
  \end{matrix}\right].
  \)
  Take $(P,P')\in\mathbb{O}^{m-|\alpha|,n-|\alpha|}(\Delta\widetilde{Z})$.
  Write $\widetilde{U}=[\overline{U}_{\alpha}\ \ \overline{U}_{\!\beta\cup\gamma}P]$
  and $\widetilde{V}=[\overline{V}_{\alpha}\ \ \overline{V}_{\!\beta\cup\gamma\cup c}P']$.
  Then,
  \[
    \Delta\Gamma=\widetilde{U}\left[\begin{matrix}
       0_{\alpha\alpha}& 0_{\alpha,\beta\cup\gamma} & 0_{\alpha c}\\
       0_{\beta\cup\gamma,\alpha} & {\rm Diag}(\sigma(\Delta\widetilde{Z}))&0_{\beta\cup\gamma,c}
     \end{matrix}\right]\widetilde{V}^{\mathbb{T}}.
  \]
  By the definitions of $\widetilde{U}$ and $\widetilde{V}$ and \eqref{XbarSVD},
  it is easy to check that $(\widetilde{U},\widetilde{V})\in\mathbb{O}^{m,n}(\overline{X})$.
  Notice that ${\rm rank}(\overline{X})=|\alpha|$. From \cite[Theorem 4]{Le13},
  we have $-\Delta\Gamma\in \partial {\rm rank}(\overline{X})$.
  Thus, $0\in \partial {\rm rank}(\overline{X})+\nu\partial\!f(\overline{X})+\mathcal{N}_{\Omega}(\overline{X})$.
  From Definition \ref{Def-Rspc}, $\overline{X}$ is a R-stationary point.

  \medskip

  Now let $\overline{X}$ be a R-stationary point of \eqref{rank-reg}
  with ${\rm rank}(\overline{X})=\overline{r}$. Suppose that $\overline{r}>1$.
  Take $\phi\in\!\mathscr{L}$ with $0\in\partial\widehat{\phi}(0)$.
  By Definition \ref{Def-Rspc}, there is
  $\Delta\Gamma\in\!\nu\partial\!f(\overline{X})+\mathcal{N}_{\Omega}(\overline{X})$
  such that $-\Delta\Gamma\in \partial {\rm rank}(\overline{X})$.
  Along with \cite[Theorem 4]{Le13}, there exists
  $(\overline{U},\overline{V})\in\mathbb{O}^{m,n}(\overline{X})$ such that
  \[
    -\Delta\Gamma=\overline{U}[{\rm Diag}(\overline{x})\ \ 0]\overline{V}^{\mathbb{T}}
    \ \ {\rm with}\ \ \overline{x}_i=0\ \ {\rm for}\ \ i=1,2,\ldots,\overline{r}.
  \]
  Next we proceed the arguments by $t^*=0$ and $t^*\ne 0$, where
  $t^*$ is same as in \eqref{phi-assump}.

  \medskip
  \noindent
  {\bf Case 1:} $t^*=0$. Take $\overline{W}:=\overline{U}_1\overline{V}_1^{\mathbb{T}}$,
  where $\overline{U}_1$ and $\overline{V}_1$ are the matrix consisting of the first
  $\overline{r}$ columns of $\overline{U}$ and $\overline{V}$, respectively. Clearly,
  $\overline{W}\in\partial\|\cdot\|_*(\overline{X})$ and
  $(\overline{U},\overline{V})\in\mathbb{O}^{m,n}(\overline{Z})$ with
  $\overline{Z}=\overline{X}+\overline{W}$.
  Let $\alpha,\beta,c,\gamma_0,\gamma_1$ be defined as before.
  Clearly, $\beta=\emptyset=\gamma_1$. Take
  \begin{equation}\label{equa-barw1}
   \overline{w}_i=\phi_{-}'(1)\ \ {\rm for}\ i\in\alpha\ \ {\rm and}\ \
   \overline{w}_i=0\in\partial\widehat{\phi}(0)\ \ {\rm for}\ i\in\gamma_0.
 \end{equation}
 Since $\phi$ is convex, from \cite[Proposition 10.19(i)]{RW98} it follows that
 \(
   \overline{w}_i\in\partial\widehat{\phi}(1)
 \)
 for $i\in\alpha$. Then $\Delta W\!=\overline{U}[{\rm Diag}(\overline{w})\ \ 0]\overline{V}^{\mathbb{T}}
  \!\in\partial(\widehat{\Phi}\circ\sigma)(\overline{W})$.
  Let $\Delta\widetilde{\Gamma}\!:=\overline{U}^{\mathbb{T}}\Delta\Gamma\overline{V}$
  and $\Delta\widetilde{W}\!:=\overline{U}^{\mathbb{T}}\Delta W\overline{V}$.
  Clearly, $\mathcal{X}(\Delta\widetilde{\Gamma}_1)=\mathcal{X}(\Delta\widetilde{W}_1)=0$
  where $\Delta\widetilde{\Gamma}_1:=\overline{U}^{\mathbb{T}}\Delta\Gamma\overline{V}_1$ and
  $\Delta\widetilde{W}_1:=\overline{U}^{\mathbb{T}}\Delta W\overline{V}_1$ with
  $\overline{V}_1$ being the matrix consisting of the first $m$ columns of $\overline{V}$.
  Together with $\Theta_2$ and $\Sigma_2$ defined as in Section \ref{sec2.3},
  it is immediate to verify that $(-\Delta\widetilde{\Gamma},-\Delta\widetilde{W})$ satisfies
  \begin{subequations}
  \begin{align*}
  \Theta_1\circ\mathcal{S}(\Delta\widetilde{W}_1)+\Theta_2\circ \mathcal{S}(\Delta\widetilde{\Gamma}_1)
    +\Sigma_1\circ \mathcal{X}(\Delta\widetilde{W}_1)+\Sigma_2\circ \mathcal{X}(\Delta\widetilde{\Gamma}_1)=0,\\
   (\Delta\widetilde{\Gamma})_{\alpha c}+(\Omega_3)_{\alpha c}\circ
   [(\Delta\widetilde{W})_{\alpha c}-(\Delta\widetilde{\Gamma})_{\alpha c}]=0,
   \ (\Delta\widetilde{W})_{\beta c}=0,\,\ (\Delta\widetilde{W})_{\gamma c}=0.
 \end{align*}
 \end{subequations}
 Since $\beta=\emptyset$, from Lemma \ref{coderivative} it follows that
 $(-\Delta\Gamma,-\Delta W)\in\mathcal{N}_{{\rm gph}\,\partial\|\cdot\|_*}(\overline{X},\overline{W})$,
 i.e., $-\Delta\Gamma\in\!D^*\partial\|\cdot\|_*(\overline{X}|\overline{W})(\Delta W)$.
 By Definition \ref{Def-Mspc}, $\overline{X}$ is M-stationary associated to $\phi$.

 \medskip
 \noindent
 {\bf Case 2:} $t^*\ne 0$. Now $t^*\in (0,1)$. Take
  $\overline{W}:=\overline{U}_1\overline{V}_1^{\mathbb{T}}+t^*\overline{U}_2\overline{V}_2^{\mathbb{T}}$,
  where $\overline{U}_2$ and $\overline{V}_2$ are the matrix consisting of the last
  $m-\overline{r}$ and $n-\overline{r}$ columns of $\overline{U}$ and $\overline{V}$, respectively.
  Clearly, $\overline{W}\in\partial\|\cdot\|_*(\overline{X})$ and
  $(\overline{U},\overline{V})\in\mathbb{O}^{m,n}(\overline{Z})$ with
  $\overline{Z}=\overline{X}+\overline{W}$.
  Let $\alpha,\beta,c$ and $\gamma=\gamma_0\cup\gamma_1$ be defined as before.
  Then $\beta=\emptyset$ and $\gamma_0=\emptyset$.
  Let $\Delta W=\overline{U}[{\rm Diag}(\overline{w})\ \ 0]\overline{V}^{\mathbb{T}}$ with
  \begin{equation}\label{equa-barw1}
   \overline{w}_i=\phi_{-}'(1)\ \ {\rm for}\ i\in\alpha\ \ {\rm and}\ \
   \overline{w}_i=0\in\partial\phi(t^*)\ \ {\rm for}\ i\in\gamma_1.
 \end{equation}
 Using the same arguments as those for Case 1 can prove that
 $\overline{X}$ is M-stationary.

 \medskip

 When $\overline{r}=0$, choose $\overline{W}=0$.
 Clearly, $\overline{W}\in\partial\|\cdot\|_*(\overline{X})$
 since $\overline{X}=0$. Write $\overline{Z}=\overline{X}+\overline{W}$.
 Then, $\alpha=\beta=\emptyset=\gamma_1$. Take $\Delta W=0$.
 Since $0\in\partial\widehat{\phi}(0)$, we have
 $\Delta W\in\partial(\widehat{\Phi}\circ\sigma)(\overline{W})$.
 Moreover, by Lemma \ref{coderivative} it is easy to check that
 \(
   D^*\partial\|\cdot\|_*(\overline{X}|\overline{W})(\Delta W)=\mathbb{R}^{m\times n}.
 \)
 Thus, $\overline{X}$ is M-stationary associated to $\phi$.
 The proof is then completed.
 \end{proof}

  To close this subsection, we provide a condition for a local minimizer of
  the MPEC \eqref{MPEC} associated with $\phi\in\!\mathscr{L}$ to be an M-stationary
  point associated to $\phi$.
 \begin{proposition}\label{necess1}
  Let $(\overline{W},\overline{X})$ be a local minimizer of the MPEC \eqref{MPEC}
  associated to $\phi\in\!\mathscr{L}$. Then $\overline{X}$ is an M-stationary point
  of the problem \eqref{rank-reg} associated to $\phi$, provided that
  \begin{equation}\label{assump-Mspc}
   \mathcal{N}_{{\rm gph}\,\mathcal{N}_{\mathbb{B}}\cap(\mathbb{R}^{m\times n}\times\Omega)}(\overline{W},\overline{X})
   \subseteq \mathcal{N}_{{\rm gph}\,\mathcal{N}_{\mathbb{B}}}(\overline{W},\overline{X})+\{0\}\times\mathcal{N}_{\Omega}(\overline{X})
  \end{equation}
  where $\mathbb{B}:=\{Z\in \mathbb{R}^{m\times n}\,|\, \|Z\|\leq 1\}$,
  and if in addition $\phi\in\!\mathscr{L}_1$, $\overline{X}$ is a $R$-stationary point.
 \end{proposition}
 \begin{proof}
  By invoking the relation \eqref{XW-equa1}, $(W,X)$ is a feasible point of \eqref{MPEC}
  if and only if $(W,X)\in{\rm gph}\,\mathcal{N}_{\mathbb{B}}\cap(\mathbb{R}^{m\times n}\times \Omega)$.
  This implies that \eqref{MPEC} can be compactly written as
  \[
   \min_{X,W\in\mathbb{R}^{m\times n}}\Big\{\nu f(X)+\widehat{\Phi}(\sigma(W))
   +\delta_{{\rm gph}\,\mathcal{N}_{\mathbb{B}}\cap(\mathbb{R}^{m\times n}\times \Omega)}(W,X)\Big\}.
  \]
  From the local optimality of $(\overline{W},\overline{X})$,
  the assumption on $f$, the Lipschitz continuity of $\widehat{\Phi}\circ \sigma$
  over the ball $\mathbb{B}$, and \cite[Theorem 10.1 $\&$ Exercise 10.10]{RW98},
  it follows that
  \[
    0\in\partial\widetilde{f}(\overline{W},\overline{X})+\partial(\widehat{\Phi}\circ\sigma)(\overline{W})\times \{0\}
     +\mathcal{N}_{{\rm gph}\,\mathcal{N}_{\mathbb{B}}\cap(\mathbb{R}^{m\times n}\times \Omega)}(\overline{W},\overline{X})
  \]
  where $\widetilde{f}(W,X)\equiv \nu f(X)$. Together with
  the inclusion \eqref{assump-Mspc} and \cite[Exercise 10.10]{RW98},
 \begin{equation}\label{inclusion1}
   0\in \{0\}\times\nu\partial f(\overline{X})+\mathcal{N}_{{\rm gph}\,\mathcal{N}_{\mathbb{B}}}(\overline{W},\overline{X})+\{0\}\times\mathcal{N}_{\Omega}(\overline{X})
      +\partial(\widehat{\Phi}\circ\sigma)(\overline{W})\times \{0\}
 \end{equation}
 which is equivalent to saying that there exists
 $(-\Delta W,\Delta X)\in\mathcal{N}_{{\rm gph}\,\mathcal{N}_{\mathbb{B}}}(\overline{W},\overline{X})$ such that
 \[
   \left\{\begin{array}{l}
           0\in \partial(\widehat{\Phi}\circ\sigma)(\overline{W}) -\Delta W\\
           0\in \nu\partial f(\overline{X})+\Delta X +\mathcal{N}_{\Omega}(\overline{X}).
          \end{array}\right.
 \]
 Notice that $(-\Delta W,\Delta X)\in\mathcal{N}_{{\rm gph}\,\mathcal{N}_{\mathbb{B}}}(\overline{W},\overline{X})$
 if and only if $\Delta X\in D^*\partial\|\cdot\|_*(\overline{X}|\overline{W})(\Delta W)$.
 So, equation \eqref{inclusion1} is equivalent to saying that
 there exists $\Delta W\in\partial(\widehat{\Phi}\circ\sigma)(\overline{W})$ such that
 \[
   0\in\nu\partial f(\overline{X})+\mathcal{N}_{\Omega}(\overline{X})
   +D^*\partial\|\cdot\|_*(\overline{X}|\overline{W})(\Delta W).
 \]
 In addition, notice that $\overline{X}\in\mathcal{N}_{\mathbb{B}}(\overline{W})$
 which is equivalent to $\partial\|\cdot\|_*(\overline{X})$ by \eqref{XW-equa2}.
 Thus, by Definition \ref{Def-Mspc}, $\overline{X}$ is an M-stationary point of
 the problem \eqref{rank-reg} associated to $\phi$. The second part is a direct
 consequence of Theorem \ref{prop-MWrelation}.
 The proof is completed.
\end{proof}
 \begin{remark}\label{optimality-Mspc}
 {\bf (i)} If $\Omega=\mathbb{R}^{m\times n}$, the inclusion \eqref{assump-Mspc}
 automatically holds. If $\Omega\subset\mathbb{R}^{m\times n}$, by \cite[Page 211]{Ioffe08}
 the inclusion \eqref{assump-Mspc} is implied by the calmness of the following multifunction
 \begin{equation}\label{HuaM}
  \mathcal{M}(Y_1,Y_2):=\Big\{(W,X)\in\mathbb{R}^{m\times n}\times\Omega\ |\
   (W,X)\in{\rm gph}\,\mathcal{N}_{\mathbb{B}}-(Y_1,Y_2)\Big\}.
 \end{equation}
  at the origin for $(\overline{W},\overline{X})$, where $(\overline{W},\overline{X})$
  is an arbitrary feasible point of the MPEC \eqref{MPEC}.

 \medskip
 \noindent
 {\bf(ii)} When $\Omega\subseteq\mathbb{S}^n_+$, together with \eqref{SDPMPEC}
 and the Lipschitz continuity of $\widehat{\Phi}\circ \sigma$ in $\mathbb{B}$,
 in order to achieve the conclusion of Proposition \ref{necess1},
 we need to replace the inclusion \eqref{assump-Mspc} by
 \[
   \mathcal{N}_{C}(\overline{W},\overline{X})
   \subseteq \mathcal{N}_{{\rm gph}\,\mathcal{N}_{\mathbb{S}_{+}^n}}(\overline{X},\overline{W}\!-I)
   +\mathcal{N}_{\mathbb{S}_{+}^n}(\overline{W})\times\mathcal{N}_{\Omega}(\overline{X})
 \]
 where $C:=\{(W,X)\in\mathbb{S}_{+}^n\times\Omega\ |\ (X,W-I)\in{\rm gph}\,\mathcal{N}_{\mathbb{S}_{+}^n}\}$.
 By invoking \cite[Page 211]{Ioffe08}, this inclusion is implied by the calmness of the following
 multifunction
  \begin{align}\label{HuaM2}
  \mathcal{M}(Y_1,Y_2):=\Big\{(W,X)\in \mathbb{S}_{+}^n\times  \Omega\,|\,
  (X,W-I)\in {\rm gph}\mathcal{N}_{\mathbb{S}^n_+}-(Y_1,Y_2)\Big\}
 \end{align}
 at the origin for $(\overline{W},\overline{X})$, where $(\overline{W},\overline{X})$
 is an arbitrary feasible point of the MPEC \eqref{SDPMPEC}.
 By the definition of calmness, it is easy to check that the calmness of
 $\mathcal{M}$ at the origin for $(\overline{W},\overline{X})$ is implied by
 that of $\mathcal{\widetilde{M}}$ in \eqref{Mdefinity} with
 $\Omega_x=\Omega$ and $\Omega_y=\mathbb{S}_{+}^n$ at the corresponding point,
 while by Theorem \ref{DireNorm} the latter holds if for any
  $0\ne H=(H_1;H_2)\in\!\mathcal{T}_{\Omega}(\overline{X})
 \times\mathcal{T}_{\mathbb{S}^n_+}(\overline{W})$ such that
 \(
    (H_1,H_2) \in\!\mathcal{T}_{{\rm gph}\mathcal{N}_{\mathbb{S}^n_+}}(\overline{X},\overline{W}-I),
 \)
 the following implication relation holds:
 \begin{align}\label{DireInclusion0}
 \left.\begin{array}{r}
 \Gamma_1\in-\mathcal{N}_{\Omega}(\overline{X};H_1), \ \ \Gamma_2\in -\mathcal{N}_{\mathbb{S}^n_+}(\overline{W};H_2)\\
 (\Gamma_1,\Gamma_2)\in
 \mathcal{N}_{{\rm gph}\mathcal{N}_{\mathbb{S}^n_+}}((\overline{X},\overline{W}-I);( H_1,H_2))
 \end{array}\right\}\Longrightarrow
 \left(\begin{matrix}
  \Gamma_1\\ \Gamma_2
  \end{matrix}\right)=0.
 \end{align}
 For the characterization of
 $\mathcal{N}_{{\rm gph}\mathcal{N}_{\mathbb{S}^n_+}}((\overline{X},\overline{W}-I);( H_1,H_2))$,
 please refer to Appendix.
 \end{remark}

 \subsection{EP-stationary points}\label{subsec3.2}

  By the definition of the function $\widehat{\Psi}$, clearly,
  the problem \eqref{Epenalty-MPEC} can be compactly written as
  \begin{align}\label{rankreg-penalty1}
   &\min_{X,W\in\mathbb{R}^{m\times n}}\!\Big\{\nu f(X)+\delta_{\Omega}(X)+\widehat{\Psi}(\sigma(W))
   +\rho(\|X\|_*\!-\langle W,X\rangle)\Big\}.
  \end{align}
  Based on this equivalent reformulation, we introduce the following stationary point.
  \begin{definition}\label{EP-spc}
   A matrix $\overline{X}\in\mathbb{R}^{m\times n}$ is said to be an EP-stationary
   point of the problem \eqref{rank-reg} associated to $\phi\in\!\mathscr{L}$
   if there exist a constant $\rho>0$ and $\overline{W}\!\in\mathbb{B}$ such that
   \begin{equation}\label{EP-equa}
     \rho \overline{X}\in\partial(\widehat{\Psi}\circ\sigma)(\overline{W})\ \ {\rm and}\ \
     0\in \nu\partial\!f(\overline{X})+\mathcal{N}_{\Omega}(\overline{X})
     +\rho\big[\partial\|\cdot\|_*(\overline{X})-\overline{W}\big].
   \end{equation}
  \end{definition}
  \begin{remark}
   By the given assumption on $f$ and the Lipschitz continuity of $\widehat{\Phi}\circ\sigma$
   in $\mathbb{B}$, if $(\overline{X},\overline{W})$ is a limiting critical point
   of the objective function of \eqref{rankreg-penalty1}, it is an EP-stationary
   point of \eqref{rank-reg}. Thus, every local optimal solution
   of \eqref{Epenalty-MPEC} is an EP-stationary point of \eqref{rank-reg}.
  \end{remark}

  The following proposition characterizes a key property of the EP-stationary point.
 \begin{proposition}\label{prop-EPspc}
  Suppose that $\overline{X}\in\mathbb{R}^{m\times n}$ is an EP-stationary point of
  \eqref{rank-reg} associated to $\phi\in\!\mathscr{L}$. Then, there exist
  $\overline{W}\in\mathbb{B}$ and $(\overline{U},\overline{V})
  \in\mathbb{O}^{m,n}(\overline{W})\cap\mathbb{O}^{m,n}(\overline{X})$ such that
  ${\rm rank}(\overline{X})\ge |\{i\ |\ \sigma_i(\overline{W})>t^*\}|$,
  and there exists $\Delta\Gamma\in\nu\partial\!f(\overline{X})+\mathcal{N}_{\Omega}(\overline{X})$
  such that
  \begin{equation}\label{EP-detaX}
    \Delta\Gamma\in\!\left\{\overline{U}\left[\begin{matrix}
                            0 & 0\\
                             0 & \rho Z
                   \end{matrix}\right]\overline{V}^{\mathbb{T}} |\,
    Z\in\mathbb{R}^{(m-|\theta_1|)\times(n-|\theta_1|)}\ {\rm with}\
    \theta_1\!=\{i\ |\,\sigma_i(\overline{W})=1\},\|Z\|\!<1\right\}.
  \end{equation}
 \end{proposition}
 \begin{proof}
  Since $\overline{X}$ is an EP-stationary point of the problem \eqref{rank-reg},
  there exist a constant $\rho>0$ and a matrix $\overline{W}\in\mathbb{B}$
  such that the inclusions in \eqref{EP-equa} hold. Define the index sets
  \[
    \theta_2:=\big\{i\ |\ \sigma_i(\overline{W})\in(0,1)\big\}\ \ {\rm and}\ \
    \theta_0:=\big\{i\ |\ \sigma_i(\overline{W})=0\big\}.
  \]
  Since $\rho\overline{X}\in\partial(\widehat{\Psi}\circ\sigma)(\overline{W})$,
  by \cite[Corollary 2.5]{Lewis95} there exists
  $(\overline{U},\overline{V})\!\in\!\mathbb{O}^{m,n}(\overline{W})$ such that
  \[
   \rho\overline{X}=\overline{U}\big[{\rm Diag}(\sigma(\overline{X}))\ \ 0\big]\overline{V}^{\mathbb{T}}
   \ \ {\rm and}\ \ \sigma(\overline{X})\in\partial \widehat{\Phi}(\sigma(\overline{W})).
  \]
  Notice that $\sigma_1(X)\ge\cdots\ge\sigma_m(X)$ with
  $\sigma_i(X)\in\partial\psi(1)$ for $i\in\theta_1$,
  $\sigma_i(X)\in\partial\psi(\sigma_i(\overline{W}))$ for $i\in\theta_2$
  and $\sigma_i(X)\in\partial\widehat{\psi}(0)$ for $i\in\theta_0$.
  Since $\partial\psi(t)\subset(0,+\infty)$ for any $t>t^*$,
  we have ${\rm rank}(\overline{X})\ge |\{i\ |\ \sigma_i(\overline{W})>t^*\}|$,
  and the first part follows. Since $(\overline{X},\overline{W})$ satisfies
  the second inclusion of \eqref{EP-equa},
  there exist $\Delta\Gamma\in\nu\partial f(\overline{X})+\mathcal{N}_{\Omega}(\overline{X})$
  such that $-\Delta\Gamma\in\rho[\partial\|\cdot\|_*(\overline{X})-\overline{W}]$.
  Write $\overline{r}={\rm rank}(\overline{X})$. From the SVD of $\overline{X}$
  in the last equation and equation \eqref{subdiff-nnorm}, we have
  \[
    \partial\|\cdot\|_*(\overline{X})
    =\Big\{\overline{U}_1\overline{V}_1^{\mathbb{T}}+\overline{U}_2\Gamma\overline{V}_2^{\mathbb{T}}\ |\
    \|\Gamma\|\le 1,\Gamma\in\mathbb{R}^{(m-\overline{r})\times(n-\overline{r})}\Big\},
  \]
  where $\overline{U}_1$ and $\overline{V}_1$ are the matrix consisting of the first
  $\overline{r}$ columns of $\overline{U}$ and $\overline{V}$, respectively,
  and $\overline{U}_2$ and $\overline{V}_2$ are the matrices consisting of
  the last $m-\overline{r}$ and $n-\overline{r}$ columns of
  $\overline{U}$ and $\overline{V}$, respectively. Together with
  $-\Delta\Gamma\in\rho[\partial\|\cdot\|_*(\overline{X})-\overline{W}]$
  and $\overline{W}=\overline{U}[{\rm Diag}(\sigma(\overline{W}))\ \ 0]\overline{V}^{\mathbb{T}}$,
  the inclusion in \eqref{EP-detaX} holds. In fact, the matrix $Z$ in the set of \eqref{EP-detaX}
  has the following form
  \[
   \left[\begin{matrix}
            {\rm Diag}(e_{\overline{r}-|\theta_1|}) & 0\\
                     0 & \Gamma
       \end{matrix}\right]-
       \left[\begin{matrix}
            {\rm Diag}(\sigma_{\theta_2'}(\overline{W})) & 0\\
                     0 &  {\rm Diag}(\sigma_{\theta_0}(\overline{W}))
       \end{matrix}\right]
   \]
  for some $\Gamma\in\mathbb{R}^{(m-\overline{r})\times(n-\overline{r})}$ with $\|\Gamma\|\le 1$,
  where $\theta_2':=\{i\in\theta_2\ |\ \sigma_i(\overline{W})\le t^*\}$.
 \end{proof}
 \begin{remark}\label{remark-EP}
  If $\overline{X}$ is an EP-stationary point of \eqref{rank-reg} and the associated
  $\overline{W}\in\mathbb{B}$ is such that $|\theta_1|={\rm rank}(\overline{X})$,
  then by Definition \ref{Def-Rspc} and \cite[Theorem 4]{Le13}
  $\overline{X}$ is a R-stationary point of \eqref{rank-reg}.
  However, when $\overline{X}$ is a R-stationary point, it is not necessarily
  EP-stationary.
 \end{remark}
 \subsection{DC-stationary point}\label{sec3.3}

  With the conjugate $\widehat{\Psi}^*$ of $\widehat{\Psi}$,
  the surrogate problem \eqref{surrogate} can be equivalently written as
  \begin{equation}\label{Lip-surrogate}
    \min_{X\in\mathbb{R}^{m\times n}}\!\Big\{\nu f(X)+\delta_{\Omega}(X)+\rho\|X\|_*-\widehat{\Psi}^*(\rho\sigma(X))\Big\}.
  \end{equation}
  By \cite[Lemma 2.3]{Lewis95}, we know that $\widehat{\Psi}^*$ is
  also absolutely symmetric. Along with its lsc and convexity,
  from \cite[Corollary 2.6]{Lewis95} it follows that
  $\widehat{\Psi}^*\circ \sigma$ is an absolutely symmetric convex function on $\mathbb{R}^m$.
  Thus, $\delta_{\Omega}(X)+\rho\|X\|_*-(\widehat{\Psi}^*\circ\sigma)(\rho X)$
  is a DC function on $\mathbb{R}^{m\times n}$. In view of this, we present
  the following DC-stationary point by the reformulation \eqref{Lip-surrogate}.
  \begin{definition}\label{Def-LSspc}
   A matrix $\overline{X}\in\mathbb{R}^{m\times n}$ is called
   a DC-stationary point of the problem \eqref{rank-reg} associated to
   $\phi\in\!\mathscr{L}$ if there exists a constant $\rho>0$ such that
   \begin{equation}\label{DC-equa}
     0\in \nu \partial f(\overline{X})+\mathcal{N}_{\Omega}(\overline{X})
       +\rho\partial\|\cdot\|_*(\overline{X})
      -\rho\partial(\widehat{\Psi}^*\circ \sigma)(\rho\overline{X}).
   \end{equation}
  \end{definition}

  When $f$ is convex, the problem \eqref{Lip-surrogate} is a DC program,
  and now $\overline{X}\in\mathbb{R}^{m\times n}$ is a DC-stationary point
  if and only if it is a critical point of the objective function of \eqref{Lip-surrogate}
  defined by Pang et al.\cite{PangMOR17}. It is worthwhile to point out that
  the limiting critical point of the objective function of \eqref{Lip-surrogate}
  is a DC-stationary point, but the converse does not hold. For the discussion
  on the DC-stationary point, the reader may refer to \cite{PangMOR17}.
  Here, we focus on the relation between the DC-stationary point and
  the EP-stationary point.
  \begin{theorem}\label{prop-LSspc}
   Let $\overline{X}$ be a DC-stationary point of \eqref{rank-reg} associated
   to $\phi\in\!\mathscr{L}$. Suppose that
   \begin{equation}\label{cond1-phi}
    \partial\psi(0)=\partial\widehat{\psi}(0)\ {\rm and}\
    \widehat{\psi}^*\ {\rm is\ differentiable\ on}\ \mathbb{R}_{+}\ {\rm with}\
    (\psi^*)'(0)=(\widehat{\psi}^*)'(0).
   \end{equation}
   Then $\overline{X}$ is an EP-stationary point. Conversely, if $\overline{X}$
   is an EP-stationary point associated to $\phi\in\Phi$ with $\phi$ nondecreasing
   on $[0,1]$, then $\overline{X}$ is necessarily a DC-stationary point.
  \end{theorem}
  \begin{proof}
   From the symmetry of $\widehat{\psi}$, it follows that
   $\widehat{\psi}^*(s)=\psi^*(|s|)$ for any $s\in\mathbb{R}$.
   Together with the given assumption, we have
   $(\widehat{\psi}^*)'(s)=(\psi^*)'(s)$ for any $s\ge 0$.
   By the differentiability of $\widehat{\psi}^*$ on $\mathbb{R}_{+}$,
   clearly, $\widehat{\Psi}^*$ is differentiable on $\mathbb{R}_{+}^m$.
   Along with its absolute symmetry and convexity,
   from \cite[Theorem 3.1]{Lewis95} it follows that $\widehat{\Psi}^*\circ \sigma$
   is differentiable in $\mathbb{R}^{m\times n}$, and consequently
   \(
     \partial(\widehat{\Psi}^*\circ\sigma)(\rho\overline{X})
     =\{\nabla(\widehat{\Psi}^*\circ\sigma)(\rho\overline{X})\}.
   \)
   Since $\overline{X}$ is a DC-stationary point of \eqref{rank-reg},
   there exists a constant $\rho>0$ such that \eqref{DC-equa} holds.
   Take $(\overline{U},\overline{V})\in\!\mathbb{O}^{m,n}(\overline{X})$. Let
   \[
     \overline{W}\!:=\overline{U}[{\rm Diag}(\overline{w})\ \ 0]\overline{V}^{\mathbb{T}}
     \ \ {\rm with}\ \ \overline{w}_i=(\psi^*)'(\rho\sigma_i(\overline{X}))
     \ \ {\rm for}\ i=1,2,\ldots,m.
   \]
   Since $\psi$ is a closed proper convex function, we have
   ${\rm range}\,\partial\psi^*\subseteq{\rm dom}\psi=[0,1]$ by \cite[Section 23]{Roc70},
   which implies that $\overline{w}_i\in[0,1]$ for $i=1,\ldots,m$
   and consequently $\|\overline{W}\|\le1$. Combining
   $\overline{w}_i=(\psi^*)'(\rho\sigma_i(\overline{X}))$
   with \cite[Corollary 23.5.1]{Roc70}, we obtain
   \[
    \rho\sigma_i(\overline{X})\in\partial\psi(\overline{w}_i)\subseteq\partial\widehat{\psi}(\overline{w}_i)
    \ \ {\rm for}\ \ i=1,2,\ldots,m,
   \]
   where the second inclusion is due to $\overline{w}_i\in[0,1]$ and
   $\partial\psi(0)=\partial\widehat{\psi}(0)$.
   By the definition of $\widehat{\Psi}$, it is not hard to obtain
   $\rho\overline{X}\in\partial(\widehat{\Psi}\circ\sigma)(\overline{W})$.
   Thus, by Definition \ref{EP-spc} and \eqref{DC-equa}, to achieve the first part
   we only need to argue that $\overline{W}=\nabla(\widehat{\Psi}^*\circ\sigma)(\rho\overline{X})$.
   Recall that $\overline{w}_i\in(\psi^*)'(\rho\sigma_i(\overline{X}))$ for each $i$
   and $(\psi^*)'(s)=(\widehat{\psi}^*)'(s)$ for all $s\ge 0$,
   we have $\overline{w}_i=(\widehat{\psi}^*)'(\rho\sigma_i(\overline{X}))$
   for each $i$. This along with the expression of $\widehat{\Psi}^*$ means that $\overline{W}=\nabla(\widehat{\Psi}^*\circ\sigma)(\rho\overline{X})$.

   \medskip

   Now suppose $\overline{X}$ is a EP-stationary point associated to
   $\phi\in\!\mathscr{L}$ with $\phi$ nondecreasing on $[0,1]$.
   Then, there exist $\rho>0$ and $\overline{W}\in\mathbb{B}$ such that
   the inclusions in \eqref{EP-equa} hold. Notice that $\psi$
   is nondecreasing and convex. Hence, $\widehat{\psi}$ is convex.
   Together with its absolute symmetry and convexity, it follows that
   $\widehat{\Psi}$ is absolutely symmetric and convex. From \cite[Corollary 2.5]{Lewis95}
   it follows that $\widehat{\Psi}\circ\sigma$ is convex over $\mathbb{R}^{m\times n}$.
   From $\rho\overline{X}\in\partial(\widehat{\Psi}\circ\sigma)(\overline{W})$,
   we get $\overline{W}\in\partial(\widehat{\Psi}\circ\sigma)^*(\rho\overline{X})$.
   By the von Neumman trace inequality, it is easy to check that
   $(\widehat{\Psi}\circ\sigma)^*=\widehat{\Psi}^*\circ\sigma$, and then
   $\overline{W}\in\partial(\widehat{\Psi}^*\circ\sigma)(\rho\overline{X})$.
   Together with the second inclusion in \eqref{EP-equa} and Definition \ref{Def-LSspc},
   we conclude that $\overline{X}$ is a DC-stationary point of \eqref{rank-reg}.
  \end{proof}

  To sum up the previous discussions, we obtain the relations as shown in Figure 1,
  where $\theta_1$ is the index set defined as in \eqref{EP-detaX} and $\mathscr{L}_2$
  denotes the family of those $\phi\in\!\mathscr{L}$ that is nondecreasing on $[0,1]$.
  We see that the set of R-stationary points is almost same as that of M-stationary points
  and includes that of EP-stationary points under the rank condition
  $|\theta_1|={\rm rank}(\overline{X})$, while for some $\phi$ the set of EP-stationary points
  coincides with that of DC-stationary points, for example, the following special $\phi$.
 \begin{example}
  Let $\phi(t)=\frac{a-1}{a+1}t^2+\frac{2}{a+1}t\ (a>1)$ for $t\in \mathbb{R}$.
  Clearly, $\phi\in\!\mathscr{L}_1\cap\mathscr{L}_2$. Also,
  \begin{equation*}
    \psi(t)=\left\{\!\begin{array}{cl}
                      \frac{a-1}{a+1}t^2+\frac{2}{a+1}t &{\rm if}\ 0\leq t\le 1\\
                       \infty & {\rm otherwise}
                \end{array}\right.{\rm and}\
  \widehat{\psi}(t)=\left\{\!\begin{array}{cl}
                      \frac{a-1}{a+1}t^2+\frac{2}{a+1}t &{\rm if}\  0\leq t\leq 1,\\
                      \frac{a-1}{a+1}t^2-\frac{2}{a+1}t &{\rm if}\ -1\leq t\le 0,\\
                       \infty & {\rm otherwise}.
                \end{array}\right.
  \end{equation*}
   After an elementary calculation, the conjugate $\psi^*$ and $\widehat{\psi}^*$
   of $\psi$ and $\widehat{\psi}$ take the form of
   \begin{equation*}
    \psi^*(\omega)=\left\{\!\begin{array}{cl}
                      0 & \textrm{if}\ \omega\leq \frac{2}{a+1},\\
                       \frac{((a+1)\omega-2)^2}{4(a^2-1)}& \textrm{if}\ \frac{2}{a+1}<\omega\le \frac{2a}{a+1},\\
                      \omega-1 & \textrm{if}\ \omega>\frac{2a}{a+1}.
                \end{array}\right.\ {\rm and}\ \
    \widehat{\psi}^*(\omega)=\psi^*(|\omega|).
   \end{equation*}
  It is easy to check that $\phi$ satisfies the conditions in \eqref{cond1-phi}
  and is nondecreasing in $[0,1]$.
 \end{example}

 \noindent
 \vspace{-0.5cm}
 \begin{figure}[htb]
   \includegraphics[width=0.95\textwidth]{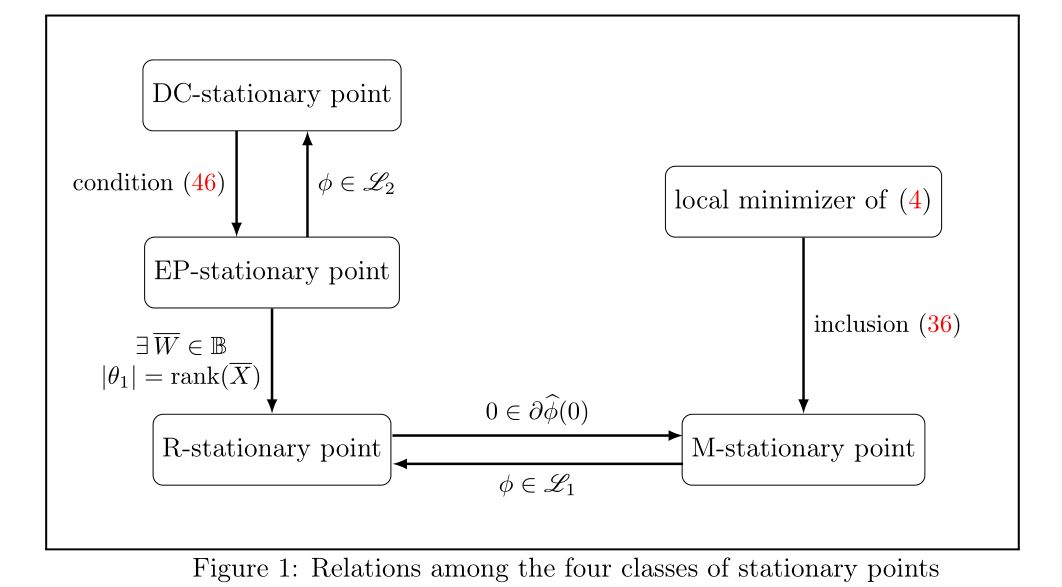}
 \end{figure}

 \section{M-stationary point of MPSCCC}\label{sec4}

  In Section \ref{sec3.1}, the MPEC \eqref{MPEC} is the key
  to characterize the M-stationary point of \eqref{rank-reg}.
  When $\Omega\subseteq\mathbb{S}_{+}^n$, it corresponds to \eqref{SDPMPEC}
  which is a special case of the following MPSCCC
 \begin{align}\label{MPSDCC}
  &\min\limits_{x\in\Omega_x,y\in\Omega_y}\varphi(x,y)\nonumber\\
  &\quad\ {\rm s.t.}\ (f(x,y),g(x,y))\in {\rm gph}\mathcal{N}_{\mathbb{S}^n_+},
 \end{align}
 where $\Omega_x\subseteq\mathbb{X}$ and $\Omega_y\subseteq\mathbb{Y}$ are the closed sets,
 $\varphi\!:\mathbb{X}\times\mathbb{Y}\to \mathbb{R}$ and
 $f,g\!:\mathbb{X}\times \mathbb{Y} \to \mathbb{S}^n$ are smooth functions.
 For this class of problems, since the Robinson CQ does not hold,
 it is common to seek an M-stationary point which is weaker
 than the classical KKT point (also called the strong stationary point).
 In this section, we shall provide a weaker condition for a local minimizer
 of \eqref{MPSDCC} to be the M-stationary point. For this purpose, we need the multifunction
 $\mathcal{\widetilde{M}}\!:\mathbb{X}\times \mathbb{Y}\times\mathbb{S}^n\times\mathbb{S}^n
 \rightrightarrows \mathbb{X}\times \mathbb{Y}$ defined as follows:
 \begin{align}\label{Mdefinity}
  \mathcal{\widetilde{M}}(u,v,\xi,\eta):=\Big\{(x,y)\in \mathbb{X}\times \mathbb{Y}\ |\
  u\in -x+ \Omega_{x},v\in -y+ \Omega_{y}, \nonumber\\
  \qquad(\xi,\eta)\in -(f(x,y),g(x,y))+{\rm gph}\mathcal{N}_{\mathbb{S}^n_+}\Big\}.
 \end{align}
 By \cite[Proposition 2.1 and  Theorem 2.1]{DingSY14}, it is immediate
 to have the following result.
 \begin{theorem}\label{theorem41}
 Let $(\overline{x},\overline{y})$ be a local minimizer of \eqref{MPSDCC}.
 If the perturbed mapping $\mathcal{\widetilde{M}}$ is calm at the origin
 for $(\overline{x},\overline{y})$, then $(\overline{x},\overline{y})$
 is an M-stationary point of the problem \eqref{MPSDCC}.
 \end{theorem}

 By \cite[Corollary 1]{GK16}, one may achieve the calmness of $\widetilde{M}$
 at the origin for $(\overline{x},\overline{y})$ by the directional
 limiting normal cone to ${\rm gph}\mathcal{N}_{\mathbb{S}^n_+}$. That is,
 the following result holds.
\begin{theorem}\label{DireNorm}
 Consider an arbitrary $(\overline{x},\overline{y})\in \mathcal{\widetilde{M}}(0,0,0,0)$.
 If for any $0\ne w=(w_1;w_2)\in\!\mathcal{T}_{\Omega_{x}}(\overline{x})
 \times\mathcal{T}_{\Omega_{y}}(\overline{y})$ such that
 \(
   \left(\begin{matrix}f'(\overline{x},\overline{y})\\
   g'(\overline{x},\overline{y})\end{matrix}\right)w
  \in\!\mathcal{T}_{{\rm gph}\mathcal{N}_{\mathbb{S}^n_+}}(f(\overline{x},\overline{y}),g(\overline{x},\overline{y})),
 \)
 the implication holds:
\begin{align}\label{DireInclusion}
 \left.\begin{array}{r}
 d_1\in\mathcal{N}_{\Omega_{x}}(\overline{x};w_1),\,
 d_1+\nabla_{x} f(\overline{x},\overline{y})\Lambda+\nabla_{x} g(\overline{x},\overline{y})\Delta=0  \\
  d_2\in\mathcal{N}_{\Omega_{y}}(\overline{y};w_2),\,
 d_2+\nabla_{y} f(\overline{x},\overline{y})\Lambda+\nabla_{y}g(\overline{x},\overline{y})\Delta =0\\
  \widetilde{w}=( f _x'(\overline{x},\overline{y})w_1+f_{y}'(\overline{x},\overline{y})w_2,
  g_x'(\overline{x},\overline{y})w_1+g_{y}'(\overline{x},\overline{y})w_2)\\
 (\Lambda,\Delta)\in\mathcal{N}_{{\rm gph}\mathcal{N}_{\mathbb{S}^n_+}}((f(\overline{x},\overline{y}),g(\overline{x},\overline{y}));\widetilde{w})
 \end{array}\right\}\Longrightarrow
 \left(\begin{matrix}
  d_1\\ d_2\\ \Lambda\\ \Delta
  \end{matrix}\right)=0,
 \end{align}
 then the multifunction $\widetilde{M}$ is calm at the origin for $(\overline{x},\overline{y})$.
 \end{theorem}
 \begin{remark}\label{PSDremark}
  {\bf(i)} Notice that
 $(\Delta x, \Delta y)\in D\mathcal{\widetilde{M}}((0,0,0,0)|(\overline{x},\overline{y}))
 (\Delta u,\Delta v,\Delta \xi, \Delta \eta)$ iff
 \begin{subnumcases}{}\label{M-tagent1}
  \Delta\omega_1:=f_x'(\overline{x},\overline{y})\Delta x+f_{y}'(\overline{x},\overline{y})\Delta y+\Delta\xi,\\
  \Delta\omega_2:=g_x'(\overline{x},\overline{y})\Delta x+g_{y}'(\overline{x},\overline{y})\Delta y+\Delta \eta,\\
  \Delta u+\Delta x\in \mathcal{T}_{\Omega_x}(\overline{x}), \Delta v+\Delta y\in \mathcal{T}_{\Omega_y}(\overline{y}),\\
  \Delta\omega_2\in D\mathcal{N}_{\mathbb{S}^n_+}(f(\overline{x},\overline{y})|g(\overline{x},\overline{y}))(\Delta\omega_1).
 \end{subnumcases}
 Together with Lemma \ref{chara-icalm}, there is no nonzero $ w=(w_1;w_2)\in
 (\mathcal{T}_{\Omega_{x}}(\overline{x})\times\mathcal{T}_{\Omega_{y}}(\overline{y}))$
 such that
 \(
   \left(\begin{matrix}f'(\overline{x},\overline{y})\\
   g'(\overline{x},\overline{y})\end{matrix}\right)w
  \in \mathcal{T}_{{\rm gph}\mathcal{N}_{\mathbb{S}^n_+}}(f(\overline{x},\overline{y}),g(\overline{x},\overline{y}))
 \)
 if and only if $\mathcal{\widetilde{M}}$ is isolated calmness
 at the origin for $(\overline{x},\overline{y})$. Thus, Theorem \ref{DireNorm}
 is stating that if $\mathcal{\widetilde{M}}$ is not isolated calm and
 the implication in \eqref{DireInclusion} holds, then $\mathcal{\widetilde{M}}$
 is necessarily calm at the origin for $(\overline{x},\overline{y})$.

 \medskip
 \noindent
 {\bf(ii)} Notice that
 $(\Delta u,\Delta v,\Delta \xi, \Delta \eta)\in D^*\mathcal{\widetilde{M}}((\overline{x},\overline{y})|(0,0,0,0))
 (-\Delta x, -\Delta y)$ if and only if
 \begin{subnumcases}{}\label{M-normal1}
 \Delta x\in \nabla_x f(\overline{x},\overline{y})\Delta \xi+\nabla_{x} g(\overline{x},\overline{y})\Delta \eta+\Delta u,
 \ \Delta u\in  \mathcal{N}_{\Omega_{x}}(\overline{x}), \\
 \label{M-normal2}
 \Delta y\in  \nabla_y f(\overline{x},\overline{y})\Delta \xi+\nabla_{y} g(\overline{x},\overline{y})\Delta \eta+\Delta v,
 \ \Delta v\in  \mathcal{N}_{\Omega_{y}}(\overline{y}),\\
 \label{M-normal3}
 \Delta \xi\in D^*\mathcal{N}_{\mathbb{S}^n_+}(f(\overline{x},\overline{y})|g(\overline{x},\overline{y}))(-\Delta \eta).
 \end{subnumcases}
 Together with Lemma \ref{chara-Aubin}, the Aubin property of $\mathcal{\widetilde{M}}$ is equivalent
 to the implication
 \begin{align}\label{AubinInclusion}
 \left.\begin{array}{r}
 0\in \nabla_x f(\overline{x},\overline{y})\Delta \xi+\nabla_{x} g(\overline{x},\overline{y})\Delta \eta
 +\Delta u\\
 0\in  \nabla_y f(\overline{x},\overline{y})\Delta \xi+\nabla_{y} g(\overline{x},\overline{y})\Delta \eta
 +\Delta v\\
 \Delta u\in \mathcal{N}_{\Omega_{x}}(\overline{x}), \Delta v\in \mathcal{N}_{\Omega_{y}}(\overline{y}) \\
 (\Delta\xi,\Delta \eta)\in \mathcal{N}_{{\rm gph}\mathcal{N}_{\mathbb{S}^n_+}}(f(\overline{x},\overline{y}),g(\overline{x},\overline{y}))
 \end{array}
\right\}\Longrightarrow (\Delta u, \Delta v,\Delta \xi, \Delta \eta)=0.
\end{align}
 Since $\mathcal{N}_{{\rm gph}\mathcal{N}_{\mathbb{S}^n_+}}((f(\overline{x},\overline{y}),g(\overline{x},\overline{y}));(d_1,d_2))
 \subset\!\mathcal{N}_{{\rm gph}\mathcal{N}_{\mathbb{S}^n_+}}(f(\overline{x},\overline{y}),g(\overline{x},\overline{y}))$
 for $(d_1,d_2)\in\!\mathbb{S}^n\times \mathbb{S}^n$, the implication in \eqref{DireInclusion}
 is weaker than the one in \eqref{AubinInclusion} which is precisely the M-stationary
 point condition given in \cite[Theorem 6.1(i)]{DingSY14}. For the characterization
 of the directional normal cone $\mathcal{N}_{{\rm gph}\mathcal{N}_{\mathbb{S}^n_+}}((f(\overline{x},\overline{y}),g(\overline{x},\overline{y}));(d_1,d_2))$,
 the reader refers to Appendix.
 \end{remark}

  To close this section, we illustrate Theorem \ref{DireNorm} by the following special example
 \begin{align}\label{example}
  &\min\limits_{x\in\mathbb{R}^3,y\in\mathbb{R}^3}\|x\|^2+\|y\|^3\nonumber\\
  &\quad\ {\rm s.t.}\ \ (\mathcal{A}(x)+C,\mathcal{A}(y)+D)\in {\rm gph}\mathcal{N}_{\mathbb{S}_{+}^3},
 \end{align}
 where $C={\rm Diag}(1,0,0)$, $D={\rm Diag}(0,0,-1)$,
 and $\mathcal{A}\!:\mathbb{R}^3\to\mathbb{S}^3$ is the linear mapping
 \[
  \mathcal{A}(x):=\left(\begin{matrix}
                        x_1 & x_3 & x_2\\
                       x_3 & x_2 & x_1\\
                        x_2 & x_1 & x_3
                       \end{matrix}\right)\quad\ \forall x\in\mathbb{R}^3.
 \]
 Consider $\overline{x}=(0,0,0)^{\mathbb{T}}$ and $\overline{y}=(0,0,0)^{\mathbb{T}}$.
 Write $X:=f(\overline{x},\overline{y})$ and $Y:=g(\overline{x},\overline{y})$.
 Clearly, $(X,Y)=(C,D)\in {\rm gph}\mathbb{S}_{+}^3$. Moreover,
 the index sets $\alpha,\beta$ and $\gamma$ defined by \eqref{index}
 with $A=X+Y$ satisfy $\alpha=\{1\}, \beta=\{2\}$ and $\gamma=\{3\}.$
 Fix an arbitrary
 $0\ne w=(w_1;w_2)\in\mathbb{R}^3\times \mathbb{R}^3$
 with $w_1=(w_{11},w_{12},w_{13})^{\mathbb{T}}$
 and $w_2=(w_{21},w_{22},w_{23})^{\mathbb{T}}$ such that
 $(G,H)\in\!\mathcal{T}_{{\rm gph}\mathcal{N}_{\mathbb{S}^n_+}}(f(\overline{x},\overline{y}),
 g(\overline{x},\overline{y}))$, where
 $G=f_x'(\overline{x},\overline{y})w_1+f_{y}'(\overline{x},\overline{y})w_2$ and
 $H=g_x'(\overline{x},\overline{y})w_1+g_{y}'(\overline{x},\overline{y})w_2$.
 Since $(G,H)\in\!\mathcal{T}_{{\rm gph}\mathcal{N}_{\mathbb{S}^n_+}}(f(\overline{x},\overline{y}),
 g(\overline{x},\overline{y}))$, by the expressions of $f$ and $g$ it is not hard to obtain
 \[
  G=\left(\begin{matrix}
    0 & 0& w_{12}\\
   0&w_{12}&0 \\
   w_{12} & 0 &0\end{matrix}\right)\ \ {\rm and}\ \
   H= \left(\begin{matrix}
   0 & 0 & w_{22}\\
   0& w_{22} &0 \\
  w_{22} & 0 & 0
  \end{matrix}\right).
 \]
 with $0\leq w_{12}\perp w_{22}\leq 0$ and $w_{12}+w_{22}\neq 0$.
 Then $B:=\!P_{\beta}^{\mathbb{T}}(G\!+\!H)P_{\beta}=w_{12}+w_{22}$.
 Let $(d_1,d_2)\in\mathbb{R}^3\times\mathbb{R}^3$ and
 $(\Lambda,\Delta)\in\mathbb{S}^3\times\mathbb{S}^3$
 satisfy the conditions on the left hand side of \eqref{DireInclusion}.
 Since $\Omega_{x}=\Omega_{y}=\mathbb{R}^3$, we have $(d_1,d_2)=0$. Thus,
 \begin{align}\label{temp-equa31}
   &\nabla_{x} f(\overline{x},\overline{y})\Lambda+\nabla_{x} g(\overline{x},\overline{y})\Delta
   =\left(\begin{matrix}
             \Lambda_{11}+2\Lambda_{23}\\
              \Lambda_{22}+2\Lambda_{13}\\
               2\Lambda_{12}+\Lambda_{33}
        \end{matrix}\right)=0,\\
 \label{temp-equa32}
   &\nabla_{y} f(\overline{x},\overline{y})\Lambda+\nabla_{y}g(\overline{x},\overline{y})\Delta
   =\left(\begin{matrix}
        \Delta_{11}+2\Delta_{23}\\
                    \Delta_{22}+2\Delta_{13}\\
                    2\Delta_{12}+\Delta_{33}
                  \end{matrix}\right)=0.
 \end{align}

 \medskip
 \noindent
 {\bf Case 1: $w_{12}>0$.} Now we have $w_{22}=0$, and
 the index sets $\pi,\delta$ and $\nu$ defined by \eqref{pinu} with $B$
 satisfy $\pi=\{1\},\delta=\emptyset$ and $\nu=\emptyset$.
 From Theorem \ref{cc-coderivate} in Appendix, it follows that
 $(\Lambda,\Delta)\in\!\mathcal{N}_{{\rm gph}\mathcal{N}_{\mathbb{S}_{+}^3}}((X,Y);(G,H))$
 if and only if
 \[
  \Lambda=\left(\begin{matrix}
   0 & 0& \Lambda_{13}\\
   0&0&\Lambda_{23} \\
   \Lambda_{13}&\Lambda_{23}&\Lambda_{33}
   \end{matrix}\right)\ \ {\rm and}\ \
   \Delta=\left(\begin{matrix}
   \Delta_{11} & \Delta_{12}& \Delta_{13}\\
   \Delta_{12}&\Delta_{22}&0 \\
   \Delta_{13}&0&0
   \end{matrix}\right)\ {\rm with}\ \Lambda_{13}+\Delta_{13}=0.
  \]
 Together with \eqref{temp-equa31} and \eqref{temp-equa32},
 we get $\Lambda=0$ and $\Delta=0$. Thus, $(d_1,d_2,\Lambda,\Delta)=(0,0,0,0)$.

 \medskip
 \noindent
 {\bf Case 2: $w_{12}=0$.} Now we have $w_{22}<0$, and consequently
 the index sets $\pi,\delta$ and $\nu$ defined by \eqref{pinu} with $B$
 satisfy $\pi=\emptyset,\delta=\emptyset$ and $\nu=\{1\}$.
 From Theorem \ref{cc-coderivate} in Appendix, it follows that
 $(\Lambda,\Delta)\in\!\mathcal{N}_{{\rm gph}\mathcal{N}_{\mathbb{S}_{+}^3}}((X,Y);(G,H))$
 if and only if
 \[
  \Lambda=\left(\begin{matrix}
   0 & 0& \Lambda_{13}\\
   0&\Lambda_{22}&\Lambda_{23} \\
   \Lambda_{13}&\Lambda_{23}&\Lambda_{33}
   \end{matrix}\right)\ \ {\rm and}\ \
   \Delta=\left(\begin{matrix}
   \Delta_{11} & \Delta_{12}& \Delta_{13}\\
   \Delta_{12}&0&0 \\
   \Delta_{13}&0&0
   \end{matrix}\right)\ {\rm with}\ \Lambda_{13}+\Delta_{13}=0.
  \]
  Together with \eqref{temp-equa31} and \eqref{temp-equa32},
  we get $\Lambda=0$ and $\Delta=0$. Thus, $(d_1,d_2,\Lambda,\Delta)=(0,0,0,0)$.

  \medskip

  The above arguments show that the implication \eqref{DireInclusion} holds,
  and then the condition in Theorem \ref{DireNorm} is satisfied. Thus, the global
  minimizer $(\overline{x},\overline{y})$ is a M-stationary point of \eqref{example},
  but by \cite[Theorem 6.1(i)]{DingSY14} we can not judge whether
  $(\overline{x},\overline{y})$ is a M-stationary or not.

  \bigskip
  \noindent
  {\bf Acknowledgement} This work is supported by the National Natural Science
  Foundation of China under project No.11571120 and No.11701186.


 \bigskip
 \noindent
 {\large\bf Appendix:  Directional limiting normal cone to ${\rm gph}\mathcal{N}_{\mathbb{S}_{+}^n}$}

 \medskip
 \noindent
 In this part we characterize the directional limiting normal cone to
 ${\rm gph}\mathcal{N}_{\mathbb{S}_{+}^n}$. For this purpose, for a given
 $A\in\mathbb{S}^n$, we denote by $\lambda(A)\in\mathbb{R}^n$ the eigenvalue
 vector of $A$ arranged in a nonincreasing order and write
 $\mathbb{O}^n(A):=\{P\in\mathbb{O}^n\ |\ A=P{\rm Diag}(\lambda(A))P^{\mathbb{T}}\}$.

 \medskip

 Fix an arbitrary $(X,Y)\in{\rm gph}\mathcal{N}_{\mathbb{S}_{+}^n}$
 and write $A\!=X\!+Y$. Suppose that $A$ has the eigenvalue decomposition
 $A=P{\rm Diag}(\lambda(A))P^\mathbb{T}$ where $P\in\mathbb{O}^n(A)$.
 Define the index sets
 \begin{equation}\label{index}
   \alpha:=\{i\,|\,\ \lambda_i(A)>0\},\ \beta:=\{i\,|\,\ \lambda_i(A)=0\}\ \ {\rm and}\ \
   \gamma:=\{i\,|\,\ \lambda_i(A)<0\}.
 \end{equation}
 For given $G,H\in\mathbb{S}^n$, by the eigenvalues of
 $B\!:=\!P_{\beta}^{\mathbb{T}}(G\!+\!H)P_{\beta}$ we define the index sets:
 \begin{equation}\label{pinu}
  \pi\!:=\!\big\{i\in[1,|\beta|]\,|\, \lambda_i(B)\!>0\big\},
  \delta\!:=\!\big\{i\in[1,|\beta|]\,|\,\lambda_i(B)=0\big\},
  \nu\!:=\!\big\{i\in[1,|\beta|]\,|\, \lambda_i(B)\!<0\big\}.
 \end{equation}
 For the index set $\delta$, we denote the set of all partitions of $\delta$ by
 $\mathscr{P}(\delta)$. Define the set
 \[
   \mathbb{R}_{\gtrsim}^{|\delta|}:=\big\{z\in\mathbb{R}^{|\delta|}\!:\ z_1\ge\cdots\ge z_{|\delta|}\big\}.
 \]
 For any $z\in\mathbb{R}_{\gtrsim}^{|\delta|}$, let $D(z)\in\mathbb{S}^{|\delta|}$ denote
 the first generalized divided difference matrix of $h(t)=\max(t,0)$ at $z$,
 which is defined as
 \begin{equation}\label{Ddiff-matrix1}
   (D(z))_{ij}:=\left\{\begin{array}{cl}
               \!\frac{\max(0,z_i)-\max(0,z_j)}{z_i-z_j}\in[0,1]&{\rm if}\ z_i\ne z_j,\\
                 1 &{\rm if}\ z_i=z_j>0,\\
                 0 & {\rm otherwise}.
                \end{array}\right.
 \end{equation}
 Write
  \(
   \mathcal{U}_{|\beta|}
   :=\big\{\overline{\Omega}\in\mathbb{S}^{|\delta|}\!:\ \overline{\Omega}=\lim_{k\to\infty}D(z^k),\,
   z^k\to 0,\, z^k\in\mathbb{R}_{\gtrsim}^{|\delta|}\big\}.
  \)
  Let $\widehat{\Xi}_1\in\mathcal{U}_{|\delta|}$. By the definition of
  $\mathcal{U}_{|\delta|}$, there exists a partition
  $(\delta_{+},\delta_{0},\delta_{-})\in\mathscr{P}(\delta)$ such that
  \[
    \widehat{\Xi}_1=\!\left[\begin{matrix}
       E_{\delta_{+}\delta_{+}}&E_{\delta_{+}\delta_{0}}& (\widehat{\Xi}_1)_{\delta_{+}\delta_{-}}\\
       E_{\delta_{0}\delta_{+}}&0& 0\\
       (\widehat{\Xi}_1)^\mathbb{T}_{\delta_{+}\delta_{-}}&0&0
       \end{matrix}\right]\in\mathbb{S}^{|\delta|}
  \]
  where every entry of $(\widehat{\Xi}_1)_{\delta_{+}\delta_{-}}$ belongs to $[0,1]$.
  We also write $\widehat{\Xi}_2=E-\widehat{\Xi}_1$, i.e.,
  \[
    \widehat{\Xi}_2:=\!\left[\begin{matrix}
       0&0& E_{\delta_{+}\delta_{-}}\!-\!(\widehat{\Xi}_1)_{\delta_{+}\delta_{-}}\\
       0&0& E_{\delta_{0}\delta_{-}}\\
       (E_{\delta_{+}\delta_{-}}\!-\!(\widehat{\Xi}_1)_{\delta_{+}\delta_{-}})^{\mathbb{T}}
       &E_{\delta_{-}\delta_{0}}&E_{\delta_{-}\delta_{-}}
       \end{matrix}\right].
  \]
  Next we provide a characterization for the direction limiting
  normal cone to ${\rm {gph}\mathcal{N}_{\mathbb{S}^n_+}}$.
  \begin{atheorem}\label{cc-coderivate}
   Consider an arbitrary point $(X,Y)\in{\rm gph}\,\mathcal{N}_{\mathbb{S}_{+}^n}$.
   Let $A=X+Y$ have the eigenvalue decomposition as above with $\alpha,\beta,\gamma$
   defined by \eqref{index}. Then, for any given $(G,H)\in\mathbb{S}^n\times\mathbb{S}^n$
   with $\pi,\delta$ and $\nu$ defined as in \eqref{pinu} for $B:=\!P_{\beta}^{\mathbb{T}}(G+H)P_{\beta}$,
   it holds that $(X^*,Y^*)\in\!\mathcal{N}_{{\rm gph}\mathcal{N}_{\mathbb{S}_{+}^n}}((X,Y);(G,H))$
   if and only if $(G,H)$ and $(X^*,Y^*)$ satisfy
   \begin{subnumcases}{}\label{GHmatrix1}
    G=P\left[\begin{matrix}
         \widetilde{G}_{\alpha\alpha} & \widetilde{G}_{\alpha\beta}&\widetilde{G}_{\alpha\gamma}\\
        \widetilde{G}_{\alpha\beta}^{\mathbb{T}}&\widetilde{G}_{\beta\beta}& 0\\
        \widetilde{G}_{\alpha\gamma}^{\mathbb{T}}& 0 & 0
       \end{matrix}\right]P^{\mathbb{T}},\ \
    H=P\left[\begin{matrix}
        0 & 0&\widetilde{H}_{\alpha\gamma}\\
        0 &\widetilde{H}_{\beta\beta}&\widetilde{H}_{\beta\gamma}\\
        \widetilde{H}_{\alpha\gamma}^{\mathbb{T}}& \widetilde{H}_{\beta\gamma}^{\mathbb{T}} &\widetilde{H}_{\gamma\gamma}
       \end{matrix}\right]P^{\mathbb{T}},\\
   \label{GHmatrix2}
   (E_{\alpha\gamma}\!-\!\Sigma_{\alpha\gamma})\circ\widetilde{G}_{\alpha\gamma}+\Sigma_{\alpha\gamma}\circ\widetilde{H}_{\alpha\gamma}=0
   ,\ \widetilde{G}_{\beta\beta}=\Pi_{\mathbb{S}_{+}^{|\beta|}}(\widetilde{G}_{\beta\beta}\!+\!\widetilde{H}_{\beta\beta})
  \end{subnumcases}
  and
 \begin{subnumcases}{}\label{XYstar1}
    X^*\!=P\left[\begin{matrix}
           0 & 0 & \widetilde{X}_{\alpha\gamma}^*\\
           0 & \widetilde{X}_{\beta\beta}^*&\widetilde{X}_{\beta\gamma}^*\\
          (\widetilde{X}_{\alpha\gamma}^*)^{\mathbb{T}}&(\widetilde{X}_{\beta\gamma}^*)^{\mathbb{T}}
          &\widetilde{X}_{\gamma\gamma}^*
          \end{matrix}\right]P^{\mathbb{T}},\
    Y^*\!=P\left[\begin{matrix}
           \widetilde{Y}_{\alpha\alpha}^* &  \widetilde{Y}_{\alpha\beta}^* & \widetilde{Y}_{\alpha\gamma}^*\\
           (\widetilde{Y}_{\alpha\beta}^*)^{\mathbb{T}} & \widetilde{Y}_{\beta\beta}^*&0\\
          (\widetilde{Y}_{\alpha\gamma}^*)^{\mathbb{T}}&0 & 0
          \end{matrix}\right]P^{\mathbb{T}},\\
  \label{XYstar2}
   \Sigma_{\alpha\gamma}\circ\widetilde{X}_{\alpha\gamma}^*
   +(E_{\alpha\gamma}\!-\!\Sigma_{\alpha\gamma})\circ(\widetilde{Y}_{\alpha\gamma}^*)=0,\
  (\widetilde{X}_{\beta\beta}^*,\widetilde{Y}_{\beta\beta}^*)\in
  \mathcal{N}_{{\rm gph}\,\mathcal{N}_{\mathbb{S}_{+}^{|\beta|}}}(0,0)\ {\rm with}\\
  \mathcal{N}_{{\rm gph}\,\mathcal{N}_{\mathbb{S}_{+}^{|\beta|}}}(0,0)
   \!=\!\bigcup_{Q\in\mathbb{O}^{|\beta|}\atop\widehat{\Xi}_1\in\mathcal{U}_{|\delta|}}
    \!\left\{(U^*,V^*)\ \Big| \left.\begin{array}{ll}
                             \Xi_1\circ(Q^{\mathbb{T}}U^*Q)+\Xi_2\circ(Q^{\mathbb{T}}V^*Q)=0,\\
                             \qquad Q_{{\delta}_0}^{\mathbb{T}}U^*Q_{{\delta}_0}\preceq 0,\
                                Q_{{\delta}_0}^{\mathbb{T}}V^*Q_{{\delta}_0}\succeq 0
                             \end{array}\right.\!\right\}
  \end{subnumcases}
  where
   \[
    \Xi_1:=\!\left[\begin{matrix}
       E_{\pi\pi} &E_{\pi\delta}& \Gamma_{\pi\nu}\\
       E_{\delta\pi} &\widehat{\Xi}_1& 0\\
       (\Gamma_{\pi\nu})^{\mathbb{T}} &0& 0\\
       \end{matrix}\right]\ \ {\rm and}\ \
     \Xi_2:=\!\left[\begin{matrix}
       0_{\pi\pi} &0_{\pi\delta}& E_{\pi\nu}- \Gamma_{\pi\nu}\\
       0_{\delta\pi} &\widehat{\Xi}_2& E_{\delta\pi}\\
       (E_{\pi\nu}- \Gamma_{\pi\nu})^{\mathbb{T}} &E_{\nu\delta}&E_{\nu\nu}\\
       \end{matrix}\right]
   \]
  with $\Gamma_{ij}=\frac{\max(0,\lambda_i(B))-\max(0,\lambda_j(B))}{\lambda_i(B)-\lambda_j(B)}$
  for each $(i,j)\in\pi\times\nu$.
 \end{atheorem}
 \begin{proof}
   ``$\Longrightarrow$''. Since $(X^*,Y^*)\in\!\mathcal{N}_{{\rm gph}\mathcal{N}_{\mathbb{S}_{+}^n}}((X,Y);(G,H))$,
  there exist sequences $t_k\downarrow 0$ and $(G^k,H^k,X^k,Y^k)\to(G,H,X^*,Y^*)$ with
  \(
    (X^k,Y^k)\in\widehat{\mathcal{N}}_{{\rm gph}\,\mathcal{N}_{\mathbb{S}_{+}^n}}(X\!+t_kG^k,Y\!+t_kH^k)
  \)
  for each $k$. Since $(X\!+t_kG^k,Y\!+t_kH^k)\in{\rm gph}\,\mathcal{N}_{\mathbb{S}_{+}^n}$ for each $k$,
  it holds that
  \[
    X+t_kG^k=\Pi_{\mathbb{S}_{+}^n}(X+Y+t_k(G^k\!+\!H^k)).
  \]
  Notice that $X=\Pi_{\mathbb{S}_{+}^n}(X\!+\!Y)$ since $(X,Y)\in{\rm gph}\,\mathcal{N}_{\mathbb{S}_{+}^n}$
  and the projection operator $\Pi_{\mathbb{S}_{+}^n}(\cdot)$ is directionally differentiable everywhere
  in the Hadamard sense. Taking the limit to the last equality, we obtain
  \(
    G=\Pi_{\mathbb{S}_{+}^n}'(X\!+\!Y;G\!+\!H).
  \)
  By the expression of $\Pi_{\mathbb{S}_{+}^n}'(X\!+\!Y;G\!+\!H)$ in \cite{Sun02},
  it follows that $G$ and $H$ satisfy \eqref{GHmatrix1}-\eqref{GHmatrix2}.
  Write $A^k\!:=(X+Y)+t_k(G^k+H^k)$. For each $k$, let $A^k$ have
  the spectral decomposition $A^k=(P^k)^{\mathbb{T}}{\rm Diag}(\lambda(A^k))P^k$
  with $P^k\in\mathbb{O}^n(A^k)$. Since $\lambda(A)=\lim_{k\to\infty}\lambda(A^k)$,
  we have $\lambda_i(A^k)>0$ for $i\in\alpha$ and $\lambda_i(A^k)<0$ for $i\in\gamma$
  when $k$ is sufficiently large, and $\lim_{k\to\infty}\lambda_i(A^k)=0$ for $i\in\beta$.
  Since $\{P^k\}$ is bounded, we may assume (if necessary taking a subsequence) that $\{P^k\}$ converges to
  $\widehat{P}\in\mathbb{O}^n(A)$. Since $P\in\mathbb{O}^n(A)$, there exists $Q\in\mathbb{O}^{|\beta|}$
  such that $\widehat{P}=[P_{\alpha}\ \ P_{\beta}Q\ \ P_{\gamma}]$. We assume
  (if necessary taking a subsequence) that there is a partition $(\beta_{+},\beta_{0},\beta_{-})$
  of $\beta$ such that
  \[
    \lambda_i(A^k)>0\ \ \forall i\in\beta_{+},\ \lambda_i(A^k)=0\ \ \forall i\in\beta_{0}\ \ {\rm and}\ \
    \lambda_i(A^k)<0\ \ \forall i\in\beta_{-}\quad{\rm for\ each}\ k.
  \]
  In addition, from \cite[Theorem 7]{Lan64} or \cite[Proposition 1.4]{Torki01}, it follows that
  \begin{equation}\label{Ak-beta}
    \lambda_i(A^k)=t_k\lambda_{l_i}\big[P_{\beta}^{\mathbb{T}}(G^k\!+\!H^k)P_{\beta}\big]+o(t_k)
    \quad\ \forall i\in\beta
  \end{equation}
  where $l_i$ is the number of eigenvalues that are equal to $\lambda_i(A)$
  but are ranked before $i$ (including $i$). Write $B^k:=P_{\beta}^{\mathbb{T}}(G^k\!+\!H^k)P_{\beta}$
  for each $k$. Since $\lambda(B)=\lim_{k\to\infty}\lambda(B^k)$, we have $\lambda_i(B^k)>0$
  for $i\in\pi$ and $\lambda_i(B^k)<0$ for $i\in\nu$ when $k$ is sufficiently large,
  and $\lim_{k\to\infty}\lambda_i(B^k)=0$ for $i\in\delta$. By further taking a subsequence if necessary,
  we may assume that there exists a partition $(\delta_{+},\delta_{0},\delta_{-})$
  of $\delta$ such that for each $k$,
  \[
    \lambda_i(B^k)>0\ \ \forall i\in\delta_{+},\ \lambda_i(B^k)=0\ \ \forall i\in\delta_{0}\ \ {\rm and}\ \
    \lambda_i(B^k)<0\ \ \forall i\in\delta_{-}.
  \]
  This means that $\pi\cup\delta_{+}=\beta_{+}-|\alpha|$ and $\nu\cup\delta_{-}=\beta_{-}-|\alpha|$,
  and then $\delta_0=\beta_0-|\alpha|$. For convenience,
  we write $\overline{\pi}=\pi+|\alpha|,\,\overline{\delta}=\delta+|\alpha|$ and
  $\overline{\nu}=\nu+|\alpha|$. Then, we have
  \[
   \{i\,|\,\ \lambda_i(A^k)>0\}=\alpha\cup\overline{\pi}\cup\overline{\delta}_{+},\
   \{i\,|\,\ \lambda_i(A^k)=0\}=\overline{\delta}_0,\
   \{i\,|\,\ \lambda_i(A^k)<0\}=\overline{\delta}_{-}\cup\overline{\nu}\cup\gamma.
  \]
   Since
  \(
    (X^k,-Y^k)\in\widehat{\mathcal{N}}_{{\rm gph}\,\mathcal{N}_{\mathbb{S}_{+}^n}}(X\!+t_kG^k,Y\!+t_kH^k),
  \)
  by \cite[Corollary 3.2]{WZZ14} it follows that
 \begin{equation}\label{temp-equa0}
   \Theta_1^k\circ\widetilde{X}^k+\Theta_2^k\circ(-\widetilde{Y}^k)=0,\ \widetilde{X}_{\overline{\delta}_{0}\overline{\delta}_{0}}^k\succeq 0
   \ \ {\rm and}\ \ \widetilde{Y}_{\overline{\delta}_{0}\overline{\delta}_{0}}^k\succeq 0
  \end{equation}
  with $\widetilde{X}^k=(P^k)^{\mathbb{T}}X^kP^k$ and $\widetilde{Y}^k=(P^k)^{\mathbb{T}}Y^kP^k$,
  where $\Theta_1^k$ and $\Theta_2^k$ take the following form
  \begin{equation}\label{Theta1k}
    \Theta_1^k=\left[\begin{matrix}
    E_{\alpha\alpha}&E_{\alpha\overline{\pi}}&E_{\alpha\overline{\delta}_{+}}&E_{\alpha\overline{\delta}_0}
    &\Sigma_{\alpha\overline{\delta}_{-}}^k&\Sigma_{\alpha\overline{\nu}}^k&\Sigma_{\alpha\gamma}^k\\
    E_{\overline{\pi}\alpha}&E_{\overline{\pi}\overline{\pi}}&E_{\overline{\pi}\overline{\delta}_{+}}&E_{\overline{\pi}\overline{\delta}_0}
    &\Sigma_{\overline{\pi}\overline{\delta}_{-}}^k&\Sigma_{\overline{\pi}\overline{\nu}}^k&\Sigma_{\overline{\pi}\gamma}^k\\
    E_{\overline{\delta}_{+}\alpha}&E_{\overline{\delta}_{+}\overline{\pi}}&E_{\overline{\delta}_{+}\overline{\delta}_{+}}
    &E_{\overline{\delta}_{+}\overline{\delta}_0}
    &\Sigma_{\overline{\delta}_{+}\overline{\delta}_{-}}^k&\Sigma_{\overline{\delta}_{+}\overline{\nu}}^k&\Sigma_{\overline{\delta}_{+}\gamma}^k\\
    E_{\overline{\delta}_0\alpha}&E_{\overline{\delta}_0\overline{\pi}}&E_{\overline{\delta}_0\overline{\delta}_{+}}&0&0&0 &0\\
    (\Sigma_{\alpha\overline{\delta}_{-}}^k)^{\mathbb{T}}&(\Sigma_{\overline{\pi}\overline{\delta}_{-}}^k)^{\mathbb{T}}
    &(\Sigma_{\overline{\delta}_{+}\overline{\delta}_{-}}^k)^{\mathbb{T}}&0&0&0&0\\
   (\Sigma_{\alpha\overline{\nu}}^k)^{\mathbb{T}}&(\Sigma_{\overline{\pi}\overline{\nu}}^k)^{\mathbb{T}}
   &(\Sigma_{\overline{\delta}_{+}\overline{\nu}}^k)^{\mathbb{T}}&0 & 0&0 &0\\
    (\Sigma_{\alpha\gamma}^k)^{\mathbb{T}}&(\Sigma_{\overline{\pi}\gamma}^k)^{\mathbb{T}}
    &(\Sigma_{\overline{\delta}_{+}\gamma}^k)^{\mathbb{T}}&0 & 0&0 &0
    \end{matrix}\right]
  \end{equation}
  and
  \begin{equation}\label{Theta2k}
   \!\Theta_2^k\!=\!\left[\begin{matrix}
    0_{\alpha\alpha}&0_{\alpha\overline{\pi}}&0_{\alpha\overline{\delta}_{+}}&0_{\alpha\overline{\delta}_{0}}
    &\widetilde{\Sigma}_{\alpha\overline{\delta}_{-}}^k&\widetilde{\Sigma}_{\alpha\overline{\nu}}^k &\widetilde{\Sigma}_{\alpha\gamma}^k\\
    0_{\overline{\pi}\alpha}&0_{\overline{\pi}\overline{\pi}}&0_{\overline{\pi}\overline{\delta}_{+}}&0_{\overline{\pi}\overline{\delta}_{0}}
    &\widetilde{\Sigma}_{\overline{\pi}\overline{\delta}_{-}}^k&\widetilde{\Sigma}_{\overline{\pi}\overline{\nu}}^k
    &\widetilde{\Sigma}_{\overline{\pi}\gamma}^k\\
    0_{\overline{\delta}_{+}\alpha}&0_{\overline{\delta}_{+}\overline{\pi}}&0_{\overline{\delta}_{+}\overline{\delta}_{+}}
    &0_{\overline{\delta}_{+}\overline{\delta}_{0}} &\widetilde{\Sigma}_{\overline{\delta}_{+}\overline{\delta}_{-}}^k
    &\widetilde{\Sigma}_{\overline{\delta}_{+}\overline{\nu}}^k&\widetilde{\Sigma}_{\overline{\delta}_{+}\gamma}^k\\
    0_{\overline{\delta}_0\alpha}&0_{\overline{\delta}_{0}\overline{\pi}}&0_{\overline{\delta}_{0}\overline{\delta}_{+}}
    &0_{\overline{\delta}_0\overline{\delta}_0}&E_{\overline{\delta}_0\overline{\delta}_{-}}&E_{\overline{\delta}\overline{\nu}}&E_{\overline{\delta}\gamma}\\
    (\widetilde{\Sigma}_{\alpha\overline{\delta}_{-}}^k)^{\mathbb{T}}
    &(\widetilde{\Sigma}_{\overline{\pi}\overline{\delta}_{-}}^k)^{\mathbb{T}}
    &(\widetilde{\Sigma}_{\overline{\delta}_{+}\overline{\delta}_{-}}^k)^{\mathbb{T}}
    &E_{\overline{\delta}_{-}\overline{\delta}_{0}}&E_{\overline{\delta}_{-}\overline{\delta}_{-}}
    &E_{\overline{\delta}_{-}\overline{\nu}} &E_{\overline{\delta}_{-}\gamma} \\
    (\widetilde{\Sigma}_{\alpha\overline{\nu}}^k)^{\mathbb{T}}
    &(\widetilde{\Sigma}_{\overline{\pi}\overline{\nu}}^k)^{\mathbb{T}}
    &(\widetilde{\Sigma}_{\overline{\delta}_{+}\overline{\nu}}^k)^{\mathbb{T}}
    &E_{\overline{\nu}\overline{\delta}_{0}}&E_{\overline{\nu}\overline{\delta}_{-}}
     &E_{\overline{\nu}\overline{\nu}}& E_{\overline{\nu}\gamma}\\
    (\widetilde{\Sigma}_{\alpha\gamma}^k)^{\mathbb{T}}&(\widetilde{\Sigma}_{\overline{\pi}\gamma}^k)^{\mathbb{T}}
    &(\widetilde{\Sigma}_{\overline{\delta}_{+}\gamma}^k)^{\mathbb{T}}
    &E_{\gamma\overline{\delta}_{0}} &E_{\gamma\overline{\delta}_{-}}& E_{\gamma\overline{\nu}}&E_{\gamma\gamma}
    \end{matrix}\right]
  \end{equation}
  with $\Sigma_{ij}^k\!:=\!\frac{\max(0,\lambda_i(A^k))-\max(0,\lambda_j(A^k))}{\lambda_i(A^k)-\lambda_j(A^k)}$
  and $\widetilde{\Sigma}_{ij}^k\!:=E_{ij}^k-\!\Sigma_{ij}^k$.
  It is immediate to have that
  \[
    \widetilde{X}^k\to
    \!\left[\begin{matrix}
      \widetilde{X}_{\alpha\alpha}^*& \widetilde{X}_{\alpha\beta}^*Q&\widetilde{X}_{\alpha\gamma}^*\\
      Q^{\mathbb{T}}(\widetilde{X}_{\alpha\beta}^*)^{\mathbb{T}}& Q^{\mathbb{T}}\widetilde{X}_{\beta\beta}^*Q&Q^{\mathbb{T}}\widetilde{X}_{\beta\gamma}^*\\
      (\widetilde{X}_{\alpha\gamma}^*)^{\mathbb{T}}& (\widetilde{X}_{\beta\gamma}^*)^{\mathbb{T}}Q&\widetilde{X}_{\gamma\gamma}^*\\
      \end{matrix}\right]\ {\rm and}\
    \widetilde{Y}^k\to
    \!\left[\begin{matrix}
      \widetilde{Y}_{\alpha\alpha}^*& \widetilde{Y}_{\alpha\beta}^*Q&\widetilde{Y}_{\alpha\gamma}^*\\
      Q^{\mathbb{T}}(\widetilde{Y}_{\alpha\beta}^*)^{\mathbb{T}}& Q^{\mathbb{T}}\widetilde{Y}_{\beta\beta}^*Q&Q^{\mathbb{T}}\widetilde{Y}_{\beta\gamma}^*\\
      (\widetilde{Y}_{\alpha\gamma}^*)^{\mathbb{T}}& (\widetilde{Y}_{\beta\gamma}^*)^{\mathbb{T}}Q&\widetilde{Y}_{\gamma\gamma}^*\\
      \end{matrix}\right].
  \]
  By the expression of $\Sigma_{ij}^k$ for
  $(i,j)\in(\alpha\cup\overline{\pi}\cup\overline{\delta}_{+})\times(\overline{\delta}_{-}\cup\overline{\nu}\cup\gamma)$
  and equation \eqref{Ak-beta},
  \begin{align*}
    \lim_{k\to\infty}\Sigma_{\alpha\overline{\delta}_{-}}^k=E_{\alpha\overline{\delta}_{-}},\
    \lim_{k\to\infty}\Sigma_{\alpha\overline{\nu}}^k\!=E_{\alpha\overline{\nu}},\
    \lim_{k\to\infty}\Sigma_{\alpha\gamma}^k\!=\Sigma_{\alpha\gamma},
     \lim_{k\to\infty}\Sigma_{\overline{\delta}_{+}\overline{\nu}}^k\!=0_{\overline{\delta}_{+}\overline{\nu}},\\
    \lim_{k\to\infty}\Sigma_{\overline{\pi}\overline{\delta}_{-}}^k\!=E_{\overline{\pi}\overline{\delta}_{-}},\
     \lim_{k\to\infty}\Sigma_{\overline{\pi}\overline{\nu}}^k\!=\Sigma_{\overline{\pi}\overline{\nu}},\
    \lim_{k\to\infty}\Sigma_{\overline{\pi}\gamma}^k\!=0_{\overline{\pi}\gamma},
    \lim_{k\to\infty}\Sigma_{\overline{\delta}_{+}\gamma}^k\!=0_{\overline{\delta}_{+}\gamma},
  \end{align*}
  where $\Sigma_{ij}=\frac{\max(0,\lambda_{i-|\alpha|}(B))-\max(0,\lambda_{j-|\alpha|}(B))}{\lambda_{i-|\alpha|}(B)-\lambda_{j-|\alpha|}(B)}$
  for $(i,j)\in\overline{\pi}\times\overline{\nu}$.
  Thus, there exists $\widehat{\Xi}_1\in\mathcal{U}_{|\delta|}$ such that
  $\widehat{\Xi}_1$ and the associated matrices $\widehat{\Xi}_2$, $\Xi_1$ and $\Xi_2$ satisfy
  \[
    \lim_{k\to\infty}\Theta_1^k=
    \Theta_1+\left[\begin{matrix}
     0_{\alpha\alpha}&0_{\alpha\beta}&0_{\alpha\gamma}\\
     0_{\beta\alpha}&\Xi_1&0\\
     0_{\alpha\gamma}&0&0\\
     \end{matrix}\right]\ \ {\rm and}\ \
    \lim_{k\to\infty}\Theta_2^k=
     \Theta_2+\left[\begin{matrix}
     0_{\alpha\alpha}&0_{\alpha\beta}&0_{\alpha\gamma}\\
     0_{\beta\alpha}&\Xi_2&0\\
     0_{\alpha\gamma}&0&0\\
     \end{matrix}\right].
  \]
  Taking the limit $k\to\infty$ to \eqref{temp-equa0},
  we have that $(X^*,Y^*)$ satisfies the condition \eqref{XYstar1}-\eqref{XYstar2}.

  \medskip
  \noindent
 ``$\Longleftarrow$''. Let $(G,H)$ and $(X^*,Y^*)$ satisfy \eqref{GHmatrix1}-\eqref{GHmatrix2}
 and \eqref{XYstar1}-\eqref{XYstar2}, respectively. We shall prove that there exist sequences
 $t_k\downarrow 0$ and  $(G^k,H^k,X^k,Y^k)\to(G,H,X^*,Y^*)$ with
  \(
    (X^k,-Y^k)\in\widehat{\mathcal{N}}_{{\rm gph}\,\mathcal{N}_{\mathbb{S}_{+}^n}}(X\!+t_kG^k,Y\!+t_kH^k)
  \)
  for each $k$. Let $B$ have the spectral decomposition
  $B=U{\rm Diag}(\lambda(B))U^{\mathbb{T}}$ and write $\widehat{P}=[P_{\alpha}\ \ P_{\beta}U\ \ P_{\gamma}]$.
  Since $(\widetilde{X}_{\beta\beta}^*,-\widetilde{Y}_{\beta\beta}^*)\in
  \mathcal{N}_{{\rm gph}\,\mathcal{N}_{\mathbb{S}_{+}^{|\beta|}}}(0,0)$,
  there exist an orthogonal matrix $Q\in\mathbb{O}^{|\beta|}$ and
  $\widehat{\Xi}_1\in\mathcal{U}_{|\delta|}$ such that
  \begin{equation}\label{temp-equa01}
   \Xi_1\circ Q^{\mathbb{T}}\widetilde{X}_{\beta\beta}^*Q
   +\Xi_2\circ(-Q^{\mathbb{T}}\widetilde{Y}_{\beta\beta}^*Q)=0,\,
   Q_{\delta_{0}}^{\mathbb{T}}\widetilde{X}_{\beta\beta}^*Q_{\delta_{0}}\succeq 0
   \ \ {\rm and}\ \ Q_{\delta_{0}}^{\mathbb{T}}\widetilde{Y}_{\beta\beta}^*Q_{\delta_{0}}\succeq 0.
  \end{equation}
  Since $\widehat{\Xi}_1\in\mathcal{U}_{|\delta|}$, we know that there exists
  a sequence $\{z^k\}\in\mathbb{R}_{\gtrsim}^{|\delta|}$ converging to $0$
  such that $\widehat{\Xi}_1=\lim_{k\to\infty}D(z^k)$. Without loss of generality,
  we can assume that there exists a partition $(\delta_{+},\delta_{0},\delta_{-})\in\mathscr{P}(\delta)$
  such that for all $k$,
  \[
    z_i^k>0\quad \forall i\in\delta_{+},\quad z_i^k=0\quad \forall i\in\delta_{0}
    \ \ {\rm and}\ \ z_i^k<0\quad \forall i\in\delta_{-}.
  \]
  For each $k$, with $x^k=[\lambda_{\pi}(B);z_{\delta_{+}}^k;0_{\delta_{0}};0_{\delta_{-}};0_{\nu}]$
  and $y^k=[0_{\pi};0_{\delta_{+}};0_{\delta_{0}};z_{\delta_{-}}^k;\lambda_{\nu}(B)]$, define
  \[
    \widehat{G}^k=\widehat{P}\left[\begin{matrix}
       \widetilde{G}_{\alpha\alpha}&\widetilde{G}_{\alpha\beta}&\widetilde{G}_{\alpha\gamma}\\
       \widetilde{G}_{\alpha\beta}^{\mathbb{T}}&{\rm Diag}(x^k)&0\\
       \widetilde{G}_{\alpha\gamma}^{\mathbb{T}}&0& 0
       \end{matrix}\right]\widehat{P}^{\mathbb{T}}\ \ {\rm and}\ \
    \widehat{H}^k=\widehat{P}\left[\begin{matrix}
       0_{\alpha\alpha}&0_{\alpha\beta}&\widetilde{H}_{\alpha\gamma}\\
       0_{\beta\alpha}&{\rm Diag}(y^k)&\widetilde{H}_{\beta\gamma}\\
       \widetilde{H}_{\alpha\gamma}^{\mathbb{T}}&\widetilde{H}_{\beta\gamma}^{\mathbb{T}}& \widetilde{H}_{\gamma\gamma}
    \end{matrix}\right]\widehat{P}^{\mathbb{T}}.
  \]
  Clearly, for each $k$, $(\widehat{G}^k,\widehat{H}^k)$ satisfies equation
  \eqref{GHmatrix1}-\eqref{GHmatrix2}, which is equivalent to saying
  that $(\widehat{G}^k,\widehat{H}^k)\in\mathcal{T}_{{\rm gph}\,\mathcal{N}_{\mathbb{S}_{+}^n}}(X,Y)$.
  Thus, for each $k$, there exist $t_{k_j}\downarrow 0$ and $(\widehat{G}^{k_j},\widehat{H}^{k_j})\to(\widehat{G}^k,\widehat{H}^k)$
  as $j\to\infty$ such that $(X,Y)+t_{k_j}(\widehat{G}^{k_j},\widehat{H}^{k_j})\in{\rm gph}\,\mathcal{N}_{\mathbb{S}_{+}^n}$ for each $j$.
  By this, we can find sequences $t_k\downarrow 0$ and $(G^k,H^k)\to(G,H)$ such that
  $(X+t_kG^k,Y+t_kH^k)\in{\rm gph}\,\mathcal{N}_{\mathbb{S}_{+}^n}$ for each $k$.
  For each $k$, write $A^k=X+Y+t_k(G^k+H^k)$ and define $\Theta_1^k\in\mathbb{S}^n$
  and $\Theta_2^k\in\mathbb{S}^n$ as in \eqref{Theta1k} and \eqref{Theta2k}, respectively,
  except that $\Sigma_{\alpha\gamma}^k$ is replaced by $\Sigma_{\alpha\gamma}$.
  By the proof of the previous necessity, for all sufficiently large $k$
  (if necessary taking a subsequence of $\{A^k\}$),
  \[
    \{i\,|\, \lambda_i(A^k)>0\}=\alpha\cup\overline{\pi}\cup\overline{\delta}_{+},\,
    \{i\,|\,\lambda_i(A^k)=0\}=\overline{\delta}_{0},\
    \{i\,|\, \lambda_i(A^k)<0\}=\overline{\delta}_{-}\cup\overline{\nu}\cup\gamma
  \]
  where $\overline{\pi}$ and $\overline{\nu}$ are the same as before and $\overline{\delta}=\delta+|\alpha|$.
  Next for each $k$ we shall define the matrices $\widetilde{X}^k\in\mathbb{S}^n$ and
  $\widetilde{Y}^k\in\mathbb{S}^n$. Let $i,j\in\{1,2,\ldots,n\}$. If $(i,j)$ and $(j,i)$

  \medskip
  \noindent
  {\bf Case 1:} $(i,j)$ or $(j,i)\in\alpha\times(\overline{\delta}_{-}\cup\overline{\nu})$.
  In this case, we have $\widetilde{X}_{ij}^*=0$ by \eqref{XYstar1}. Define
  \begin{equation}\label{abminus}
    \widetilde{Y}_{ij}^k\equiv\widetilde{Y}_{ij}^*\ \ {\rm and}\ \
    \widetilde{X}_{ij}^k=\frac{1-\Sigma_{ij}^k}{\Sigma_{ij}^k}\widetilde{Y}_{ij}^k.
  \end{equation}
  Since $\Sigma_{ij}^k\to 1$ in this case, it immediately follows that
  $(\widetilde{X}_{ij}^k,\widetilde{Y}_{ij}^k)\to (\widetilde{X}_{ij}^*,\widetilde{Y}_{ij}^*)$.

  \medskip
  \noindent
  {\bf Case 2:} $(i,j)$ or $(j,i)\in\overline{\pi}\times\gamma$. Now we have
  $\widetilde{Y}_{ij}^*=0$ by equation \eqref{XYstar1}. Define
  \begin{equation}\label{bpg}
    \widetilde{X}_{ij}^k\equiv\widetilde{X}_{ij}^*\ \ {\rm and}\ \
    \widetilde{Y}_{ij}^k=\frac{\Sigma_{ij}^k}{1-\Sigma_{ij}^k}\widetilde{X}_{ij}^k.
  \end{equation}
  Notice that $\Sigma_{ij}^k\to 0$ in this case. It immediately follows that
  $(\widetilde{X}_{ij}^k,\widetilde{Y}_{ij}^k)\to (\widetilde{X}_{ij}^*,\widetilde{Y}_{ij}^*)$.

   \medskip
  \noindent
  {\bf Case 3:} $(i,j)$ or $(j,i)\in\overline{\pi}\times\overline{\nu}$. Now
  $\Sigma_{ij}^k=\frac{\lambda_i(A^k)}{\lambda_i(A^k)^k-\lambda_j(A^k)}
  =\frac{\lambda_{l_i}(G^k+H^k)}{\lambda_{l_i}(G^k+H^k)-\lambda_{l_j}(G^k+H^k)}$ for each $k$.
  Together with the definition of $\Xi_1$, $\Sigma_{ij}^k\to(\Xi_1)_{i'j'}$ with $i'=i-|\alpha|$ and $j'=j-|\alpha|$.

  \medskip
  \noindent
  {\bf Subcase 3.1:} $(\Xi_1)_{ij}\ne 1$. Then $\Sigma_{ij}^k\ne 1$ for
  all sufficiently large $k$. We define
  \begin{equation}\label{bpbm1}
    \widetilde{X}_{ij}^k\equiv Q_i^{\mathbb{T}}\widetilde{X}_{\beta\beta}^*Q_j\ \ {\rm and}\ \
    \widetilde{Y}_{ij}^k=\frac{\Sigma_{ij}^k}{1-\Sigma_{ij}^k}\widetilde{X}_{ij}^k.
  \end{equation}
  From equation \eqref{temp-equa01}, it follows that
  $\widetilde{Y}_{ij}^k\to\frac{(\Xi_1)_{ij}}{1-(\Xi_1)_{ij}}Q_i^{\mathbb{T}}\widetilde{X}_{\beta\beta}^*Q_j
  =Q_i^{\mathbb{T}}\widetilde{Y}_{\beta\beta}^*Q_j$.

   \medskip
  \noindent
  {\bf Subcase 3.2:} $(\Xi_1)_{ij}=1$. Since $\Sigma_{ij}^k\ne 0$ for all sufficiently large $k$, we define
  \begin{equation}\label{bpbm2}
    \widetilde{Y}_{ij}^k\equiv Q_i^{\mathbb{T}}\widetilde{Y}_{\beta\beta}^*Q_j\ \ {\rm and}\ \
    \widetilde{X}_{ij}^k=\frac{1-\Sigma_{ij}^k}{\Sigma_{ij}^k}\widetilde{Y}_{ij}^k.
  \end{equation}
  From equation \eqref{temp-equa01}, it follows that
  $\widetilde{X}_{ij}^k\to\frac{1-(\Xi_1)_{ij}}{(\Xi_1)_{ij}}Q_i^{\mathbb{T}}\widetilde{Y}_{\beta\beta}^*Q_j
  =Q_i^{\mathbb{T}}\widetilde{X}_{\beta\beta}^*Q_j$.

  \medskip
  \noindent
  {\bf Case 4:} $(i,j)$ or $(j,i)\in(\beta\times\beta)\backslash(\overline{\pi}\times\overline{\nu})$.
  In this case we define
  \begin{equation}\label{bb}
    \widetilde{X}_{ij}^k=Q_i^{\mathbb{T}}\widetilde{X}_{\beta\beta}^*Q_j\ \ {\rm and}\ \
    \widetilde{Y}_{ij}^k\equiv Q_i^{\mathbb{T}}\widetilde{Y}_{\beta\beta}^*Q_j.
  \end{equation}

  \medskip
  \noindent
  {\bf Case 5:} $(i,j)$ or $(j,i)\notin(\alpha\times\overline{\nu})\cup(\overline{\pi}\times\gamma)\cup(\beta\times\beta)$.
  In this case we define
  \begin{equation}\label{othercase}
    \widetilde{X}_{ij}^k\equiv\widetilde{X}_{\beta\beta}^*\ \ {\rm and}\ \
    \widetilde{Y}_{ij}^k\equiv \widetilde{Y}_{\beta\beta}^*.
  \end{equation}

  Now for each $k$ we define $X^k=\widehat{P}\widetilde{X}^k\widehat{P}^{\mathbb{T}}$
  and $Y^k=\widehat{P}\widetilde{Y}^k\widehat{P}^{\mathbb{T}}$.
  Then, from \eqref{abminus}-\eqref{othercase}, it follows that
  $(\widehat{P}^{\mathbb{T}}X^k\widehat{P},\widehat{P}^{\mathbb{T}}Y^k\widehat{P})
  =(\widetilde{X}^k,\widetilde{Y}^k)\to (X^*,Y^*)$ as $k\to\infty$, and moreover,
  \[
    \Theta_1^k\circ(\widehat{P}^{\mathbb{T}}X^k\widehat{P})
    +\Theta_2^k\circ(-\widehat{P}^{\mathbb{T}}Y^k\widehat{P})=0,\quad k=1,2,\ldots,
  \]
  Moreover, from equations \eqref{bb} and the last inequalities in \eqref{temp-equa01},
  it follows that
  \[
    Q_{\beta_0}^{\mathbb{T}}\widetilde{X}^kQ_{\beta_0}\equiv
    Q_{\beta_0}^{\mathbb{T}}\widetilde{X}_{\beta\beta}^*Q_{\beta_0}\succeq 0\ \ {\rm and}\ \
    Q_{\beta_0}^{\mathbb{T}}\widetilde{Y}^kQ_{\beta_0}\equiv
    Q_{\beta_0}^{\mathbb{T}}\widetilde{Y}_{\beta\beta}^*Q_{\beta_0}\succeq 0,\quad k=1,2,\ldots.
  \]
  By \cite [Corollary 3.2]{WZZ14}, we have that
  $(-X^k,Y^k)\in\widehat{\mathcal{N}}_{{\rm gph}\,\mathcal{N}_{\mathbb{S}_{+}^n}}(\overline{X}^k,\overline{Y}^k)$
  for each $k$ with $(X^*,Y^*)=\lim_{k\to\infty}(X^k,Y^k)$. To sum up,
  there exist $t_k\downarrow 0$ and
  $(G^k,H^k,X^k,Y^k)\to(G^*,H^*,X^*,Y^*)$ with
  \(
    (-X^k,Y^k)\in\widehat{\mathcal{N}}_{{\rm gph}\,\mathcal{N}_{\mathbb{S}_{+}^n}}(\overline{X}\!+t_kG^k,\overline{Y}\!+t_kH^k)
  \)
  for each $k$.
 \end{proof}
 \end{document}